\documentclass{article}
\usepackage{here}
\usepackage[T1]{fontenc}
\usepackage[english]{babel}
\usepackage[latin1]{inputenc}
\usepackage{amsfonts}
\usepackage{indentfirst}
\usepackage{mathenv}
\usepackage{amsthm}
\usepackage{here}
\usepackage{amsmath}
\usepackage{color}
\usepackage{graphicx}
\usepackage{epsfig}
\usepackage{amssymb}
\usepackage{amstext}
\usepackage{color}
\usepackage{hyperref} % Pour PdfLatex
\usepackage{geometry}
\numberwithin{equation}{section}
\newtheorem{Theorem}{Theorem}
\newtheorem{Lemma}{Lemma}

\newtheorem{Proposition}{Proposition}

\allowdisplaybreaks[1]
\theoremstyle{definition}
\newtheorem{Definition}{Definition}
\theoremstyle{remark}
\newtheorem{Remark}{Remark}

\newcommand{\R}{\mathbb{R}}

\newcommand{\cqfd}
{%
\mbox{}%
\nolinebreak%
\hfill%
\rule{2mm}{2mm}%
\medbreak%
\par%
}

\date{\empty}
\begin{document}

\author{Jean-Michel Coron\footnotemark[1], Pierre Lissy\footnotemark[1]}
\footnotetext[1]{Work
supported by ERC advanced grant 266907 (CPDENL) of the 7th Research Framework Programme (FP7).}
\title{Local null controllability of the three-dimensional Navier-Stokes system with a distributed control having two vanishing components}
\maketitle
%\ead{coron@ann.jussieu.fr}
%\address[JMC]{Institut universitaire de France et Laboratoire Jacques-Louis Lions, Universit\'{e} Pierre et Marie Curie, 4 place Jussieu, 75252 Paris Cedex 05, France}
%\address[PL]{Laboratoire Jacques-Louis Lions, Universit\'{e} Pierre et Marie Curie, 4 place Jussieu, 75252 Paris Cedex 05, France}
%\ead{lissy@ann.jussieu.fr}
\vspace{1\baselineskip}
\noindent{\textbf{Abstract}}

\vspace{1\baselineskip}
In this paper, we prove a local null controllability result for the three-dimensional
Navier-Stokes equations on a (smooth) bounded domain
of $\mathbb R^3$ with null Dirichlet boundary
conditions. The control is distributed in an arbitrarily small nonempty open subset and has two vanishing components. J.-L.~Lions and E.~Zuazua proved that the linearized system is
not necessarily null controllable even if the control is distributed on the entire domain, hence the standard linearization method fails. We use the return method together with a new algebraic method inspired by the works of M.~Gromov and previous results by  M.~Gueye.

\vspace{1\baselineskip}
\vspace{2\baselineskip}

\section{Introduction}
\subsection{Notations and statement of the theorem}
Let $T>0$, let $\Omega$ be  a nonempty bounded domain of $\R^3$ of class $C^\infty$ and let $\omega$ be a nonempty open subset of $\Omega$. We define $Q\subset \R\times \R^3$ by

$$
Q:=(0,T)\times \Omega=\{(t,x)| \, t\in (0,T) \text{ and } x \in \Omega\}
$$ and we call
$$\Sigma:=[0,T]\times\partial\Omega.$$
The current point $x\in \R^3$ is $x=(x_1,x_2,x_3)$.
The $i$-th component of a vector (or a vector field) $f$ is denoted $f^i$. The control is $u=(u^1,u^2,u^3)\in L^2 (Q)^3$. We require that
the support of $u$ is included in $\omega$, which is our control domain. We impose that two components of $u$ vanish, for example the first two:
\begin{equation}
u^1= 0 \text{ and } u^2=0 \text{ in } Q,
\label{cn}
\end{equation}
so that $u$ will be written under the form $(0,0,1_\omega v)$ with $v\in L^2(Q)$ from now on, where $1_\omega: \Omega \rightarrow \R$ is the characteristic function of $\omega$:
$$
1_\omega = 1 \text{ in } \omega, \, 1_\omega = 0 \text{ in } \Omega \setminus \omega.
$$
Let us define
$$V:=\{y\in H_0^1(\Omega)^3|\nabla \cdot y=0\}.$$
The space $V$ is equipped with the $H^1_0$-norm. Let us denote by $H$ the closure of $V$ in $L^2(\Omega)^3$. The space $H$ is equipped with the $L^2$-norm.

We are interested in the following Navier-Stokes control system:
\begin{equation}
\left \{\begin{aligned}
y_t-\Delta y+(y\cdot \nabla)y+\nabla p&=(0,0,1_\omega v) &\mbox{ in }Q,\\
\nabla \cdot y&=0&\mbox{ in }Q,\\
y&= 0&\mbox{ on }\Sigma.
\end{aligned}
\right .
\label{NS}
\end{equation}
 From \cite[Theorem 3.1, p.~282]{MR603444}, we have the following existence result: For every $y^0\in H$, there exist $y\in L^2((0,T),V)\cap L^\infty ((0,T),H)$ and $p\in L^2(Q)$ satisfying
\begin{gather}
\label{conditinity0}
y(0,\cdot)=y^0 \text{ in } \Omega
\end{gather}
such that \eqref{NS} holds.

Our main result is the following theorem,
which expresses the small-time local null-controllability of \eqref{NS}:
\begin{Theorem}
\label{th-STLC}
For every $T>0$ and for every $r >0$, there exists $\eta>0$ such that, for every $y^0\in V$ verifying $||y^0||_{H_0^1(\Omega)^3}\leqslant \eta$, there exist  $v\in L^2(Q)$ and a solution $(y,p)\in L^2((0,T),H^2(\Omega)^3\cap V)\cap L^\infty ((0,T),H^1(\Omega)^3\cap V)\times L^2(Q)$ of \eqref{NS}-\eqref{conditinity0} such that
\begin{gather}
y(T,\cdot )= 0,
\\
\label{upetit}
||v||_{L^2(Q)^3}\leqslant r,
\\
\label{ypetit}
||y||_{L^2((0,T),H^2(\Omega)^3)\cap L^\infty ((0,T),H_0^1(\Omega)^3)}\leqslant r.
\end{gather}

\end{Theorem}

\begin{Remark}\label{uniqueness} Once a control $v\in L^2(Q)$ is given, the corresponding solution $(y,p)$ of \eqref{NS}, \eqref{conditinity0} and \eqref{ypetit} given by Theorem~\ref{th-STLC} is unique (recall that for the Navier-Stokes system, the uniqueness of $(y,p)$ means that $y$ is unique and $p$ is unique up to a constant depending on the time). This comes from the uniqueness result given in \cite[Theorem 3.4, p. 297]{MR603444}: One has $$L^\infty((0,T),H^1(\Omega)^3)\subset L^8((0,T),H^1(\Omega)^3)\subset L^8((0,T),L^4(\Omega)^3)$$ thanks to a classical Sobolev embedding, and there is at most one solution $(y,p)$ of \eqref{NS} and \eqref{conditinity0} in the space
$$L^2((0,T),V)\cap L^\infty((0,T),H)\cap L^8((0,T),L^4(\Omega)^3)\times L^2(Q).$$
\end{Remark}
\begin{Remark} One observes that in Theorem~\ref{th-STLC} the initial condition $y^0$ is more regular than usual ($y^0\in H$). In fact, using the same arguments as in \cite{04FGIP} and \cite{MR2225301} (see also \cite[Remark 1]{CG1}), one can easily extend the previous theorem to small initial data in $H\cap L^4(\Omega)^3$ with a solution $(y,p)\in L^2((0,T),V)\cap L^\infty((0,T),H)\times L^2(Q)$. In this case, Remark~\ref{uniqueness} is no longer true and there might possibly exist many solutions $(y,p)$ verifying \eqref{NS} and \eqref{conditinity0} once $v$ is given.
\end{Remark}

\subsection{Some previous results}
The controllability of the two or three-dimensional Navier-Stokes equations with a distributed control has been studied in numerous papers. In general, for Navier-Stokes equations, it is relevant to consider the approximate controllability, the null controllability or the exact controllability to the trajectories, the second one being a particular case of the third one.

In \cite{Icocv}, a first result of local exact controllability to the trajectories was established
under technical conditions: $\Omega$ had to be homeomorphic to a ball, the control had to be supported in a nonempty open subset whose closure is included in $\Omega$, and the target trajectory had to be a stationary solution of the Navier-Stokes equation. Moreover, there were some technical regularity conditions for these stationary solutions. A similar result for the linearized Navier-Stokes equations was established but with the same strong conditions. Many of these hypotheses were removed in \cite{Icocv2}.

Then, it was proved in \cite{04FGIP} the local exact controllability to the trajectories with regularity conditions that were weaker and more suitable for the study of the Navier-Stokes equations. In this article, the authors also proved some exact controllability results for linearized Navier-Stokes systems, with very weak regularity conditions. The same authors proved in \cite{MR2225301} the  local exact controllability to the trajectories with a control having one vanishing component, provided that $\omega$ ``touches'' the boundary of the domain $\Omega$ in some sense. Later on it was proved in \cite{MR2503028} a local null controllability result for the Stokes system with a control having a vanishing component without the geometrical condition on $\omega$, but the authors were not able to extend it to the nonlinear Navier-Stokes system. A recent work (\cite{CG1}) improved the previous one and proved the local null controllability of the Stokes system with an additional source member by means of a control having a vanishing component, which enabled the authors to prove the local null controllability of the Navier-Stokes system for a control having a vanishing component.
In all these articles, the main points of the proof were to establish first the controllability of the linearized control system
around the target trajectory thanks to Carleman estimates on the adjoint of the linearized equation, and then to use an inverse mapping theorem or a fixed-point theorem to deal with the nonlinear system.

The natural question is then: Can we remove another component
of the control, which would be an optimal result with respect to the number of controls?
Reducing the number of components of the control
is important for applications, and have already been studied many times for linear or
parabolic systems of second order (that are quite similar to linearized Navier-Stokes systems), see for example \cite{AmmarKhodja2011555}, where a necessary and sufficient condition to control a system of coupled parabolic equations with constant coefficients and with less controls than equations is given, or \cite{MR2554977,2012-Mauffrey} for time-dependent coefficients. If the coefficients depend on the time and the space, there are no general results, in particular if we consider two coupled parabolic systems where the coupling region and the control region do not intersect (a partial result under the Geometric Control Condition is given in \cite{MR3039207}). For a recent
survey on the controllability of coupled linear
parabolic equations, see \cite{2011-ABGT-MCRF}.
\subsection{The linear test}
\label{sub-linear}
To obtain Theorem~\ref{th-STLC}, the first natural idea  is to linearize the system around $0$, i.e. to consider the Stokes control system
\begin{equation}
\left \{\begin{aligned}
y_t-\Delta y+\nabla p&=(0,0,v1_{\omega})&\mbox{ in } Q,
\\ \nabla \cdot y&=0&\mbox{ in } Q,
\\y(0,\cdot )&=y^0&\mbox{ in } \Omega,
\\y&= 0&\mbox{ on }\Sigma.
\end{aligned}
\right . \label{S}
\end{equation}
It is well-known (see for example \cite{MR2225301}, or \cite{Icocv2}) that if this linear system were null controllable (with, in addition, an arbitrary source term in a suitable space), then applying an inverse mapping theorem (for example the one presented in \cite{MR924574}) in some relevant weighted spaces,  we  would obtain that \eqref{NS} is locally null controllable around $0$. However, the linear control system \eqref{S} is in general not null controllable and not approximately controllable: In \cite{MR1371594}, it is proved that this is for example the case if $\Omega$ is a cylinder with a circular generating set and with an axis parallel to $e_3$, even if we control on the entire cylinder (the approximate controllability property holds ``generically'' with respect to the generating set of the cylinder as explained in \cite{MR1371594} though).

Since linearizing around $0$ is not relevant, we are going to use the return method, which consists in linearizing system \eqref{NS} around a  particular trajectory $(\overline y,\overline p,\overline u)$ (that we construct explicitly) verifying $\overline y(0,\cdot )=\overline y(T,\cdot )= 0$, proving that the linearized system (with a source term $f$ verifying an exponential decrease condition at time $t=T$) is null controllable, and then concluding by a usual inverse mapping argument that our system is locally null controllable. This method was introduced in \cite{92mcss} for a stabilization problem concerning nonlinear ordinary differential equations and first used in the context of partial differential equations in \cite{92cras}. The return method was already successfully used  in \cite{96cocv,96rjmp,99Fursikov-Imanuvilov} to obtain global controllability results for the Navier-Stokes equations and in \cite{MR2558423} to prove the local null controllability for the Navier-Stokes equations on the torus $\mathbb{T}_2$ when the control has one vanishing component. For more explanations about the return method and other examples of applications, see \cite[Chapter 6]{MR2302744}.

\subsection{Structure of the article and sketch of the proof of Theorem~\ref{th-STLC}}
The paper is organized as follows.
\begin{itemize}
\item In Section~\ref{secconstrbar}, according to what was explained at the end of Subsection~\ref{sub-linear}, we construct a family of explicit particular trajectories $(\overline y,\overline p,\overline u)$ of the controlled Navier-Stokes system \eqref{NS} going from $0$ at time $t=0$ to $0$ at time $t=T$. These trajectories are compactly supported in $[T/4,T]\times\omega$ and vanish exponentially at time $t=T$. Moreover, they are polynomials in space on some subcylinder of $\omega$ denoted $\mathcal C_2$, and they can be arbitrarily small. We then linearize \eqref{NS} around $(\overline y,\overline p,\overline u)$ and study the linearized equation
\begin{equation}
\label{sketch1}
\left\{\begin{aligned}
y^1_t-\Delta y^1+(\overline y\cdot \nabla)y^1 +(y\cdot \nabla){\overline y}^1+ \partial_{x_1}p&=f^1&\mbox{ in } Q,
\\
y^2_t-\Delta y^2+(\overline y\cdot \nabla)y^2 +(y\cdot \nabla){\overline y}^2+ \partial_{x_2}p&=f^2&\mbox{ in } Q,
\\
y^3_t-\Delta y^3+(\overline y\cdot \nabla)y^3 +(y\cdot \nabla){\overline y}^3+ \partial_{x_3}p&=1_\omega v+f^3&\mbox{ in } Q,
\\
\nabla \cdot y&=0&\mbox{ in }Q,
\\
y&= 0&\mbox{ on }\Sigma,
\end{aligned}\right .
\end{equation}
where $f$ is some source term in an appropriate space.
\item Section~\ref{seccontrollinearized} is devoted to proving that \eqref{sketch1} is indeed null controllable (Proposition~\ref{cor-une-seule-composante}). Subsection~\ref{secmotivation-notations} is dedicated to introducing some useful notations and the crucial Proposition~\ref{prop-alg-controllability}. This proposition explains that we can split up our proof of the null controllability of the linearized equations with a scalar control into two parts:
\item Firstly, we control the following linearized Navier-Stokes system:
\begin{gather}
\label{sketch2}
\left\{\begin{aligned}
y^{*1}_t-\Delta y^{*1}+(\overline y\cdot \nabla)y^{*1} +(y^*\cdot \nabla){\overline y}^1+ \partial_{x_1}p^*&=\mathcal{B}^1u^*+f^1&\mbox{ in } Q,
\\
y^{*2}_t-\Delta y^{*2}+(\overline y\cdot \nabla)y^{*2} +(y^*\cdot \nabla){\overline y}^2+ \partial_{x_2}p^*&=\mathcal{B}^2u^*+f^2&\mbox{ in } Q,
\\
y^{*3}_t-\Delta y^{*3}+(\overline y\cdot \nabla)y^{*3} +(y^*\cdot \nabla){\overline y}^3+ \partial_{x_3}p^*&=\mathcal{B}^3u^*+f^3&\mbox{ in } Q,
\\
\nabla \cdot y^*&=0&\mbox{ in }Q,
\\
y^*&= 0&\mbox{ on }\Sigma,
\end{aligned}\right .
\end{gather}
where $\mathcal B$ is some suitable local control operator that acts on each equation.
This is the purpose of Subsection~\ref{seccontrollabilitywithB}, where the controllability of System~\eqref{sketch2} is proved thanks to the usual HUM method. More precisely, we prove an appropriate Carleman estimate with observation $\mathcal B^*$ on the adjoint equation of \eqref{sketch2} (Lemma~\ref{Carl1}), so that we create controls in the image of $\mathcal B$ thanks to the study of an appropriate Lax-Milgram type problem, which also enables us to obtain controls that are very regular in the sense that they are in weighted Sobolev spaces of high order in space and time (Proposition~\ref{cor-regular-controls}).
Let $(y^*,p^*,u^*)$ be a trajectory of \eqref{sketch2} that brings the initial condition $y^0$ to $0$ at time $T$, with a very regular $u^*$ compactly supported in space at each time in some open subset $\omega_0$ of $\mathcal C_2$ to be chosen later, and that decreases exponentially at time $t=T$. We emphasize that $(y^*,p^*)$ is less regular than $\mathcal Bu^*$ (however, it is in some weighted Sobolev space of small order) because the source term $f$ is not as regular as $\mathcal Bu^*$.
\item Secondly, we study in Subsection~\ref{secalgebraicsolvability} the following system locally on $Q_0:=[T/2,T]\times \omega_0$:
\begin{gather}
\label{sketch3}
\left\{\begin{aligned}
\tilde y^1_t-\Delta \tilde y^1+(\overline y\cdot \nabla)\tilde y^1 +(\tilde y\cdot \nabla){\overline y}^1+ \partial_{x_1}\tilde p&=-\mathcal{B}^1u^*&\mbox{ in } Q_0,
\\
\tilde y^2_t-\Delta \tilde y^2+(\overline y\cdot \nabla)\tilde y^2 +(\tilde y\cdot \nabla){\overline y}^2+ \partial_{x_2}\tilde p&=-\mathcal{B}^2u^*&\mbox{ in } Q_0,
\\
\tilde y^3_t-\Delta \tilde y^3+(\overline y\cdot \nabla)\tilde y^3 +(\tilde y\cdot \nabla){\overline y}^3+ \partial_{x_3}\tilde p&=-\mathcal{B}^3u^*+\tilde v&\mbox{ in } Q_0,
\\
\nabla \cdot \tilde y&=0&\mbox{ in }Q_0,
\end{aligned}\right .
\end{gather}
where $u^*$ has been introduced above, and where the unknowns are $(\tilde y,\tilde p,\tilde v)$. We want to prove that there exists a solution $(\tilde y,\tilde p,\tilde v)$ of \eqref{sketch3} (extended by $0$ on $[T/2,T]\times\Omega$) which has the same support as $u^*$. This seems reasonable because System~\eqref{sketch3} is analytically underdetermined: we have $5$ unknowns (the $3$ components of $\tilde y$, the pressure $\tilde p$ and the scalar control $\tilde v$) and only $4$ equations.
In fact, we prove in Proposition~\ref{prop-solvable} that is is possible to find such a  $(\tilde y,\tilde p,\tilde v)$ which can moreover be expressed as a linear combination of $u^*$ and some of its derivatives up to a certain order. This explains why we need $u^*$ to be very regular. Since $u^*$ decreases exponentially at time $T$, this is also the case for $(\tilde y,\tilde p,\tilde v)$.  The main idea behind the proof of the existence of such a $(\tilde y,\tilde p,\tilde v)$ is to consider the adjoint system of \eqref{sketch3} and to differentiate the equations appearing in this system until we get more equations than ``unknowns'', the ``unknowns'' being there the functions and all their derivatives appearing in the equations of the adjoint system.
Since Subsection~\ref{secalgebraicsolvability} is the most innovative, important, and difficult part of the article, we give some further details.
\begin{enumerate}
\item In Paragraph~\ref{subsub-adj}, we make a choice for operator $\mathcal B$ and we prove that the existence of $(\tilde y,\tilde p,\tilde v)$ can be reduced to proving the following property: There exists some $\omega_0$ and
a linear partial differential operator $\mathcal{N}:C^{\infty}(Q_0)^4
\rightarrow C^{\infty}(Q_0)^6$ such that
for every $\varphi=(\varphi^1,\varphi^2,\varphi^3,\varphi^4)\in C^{\infty}(Q_0)^4$, if
$(z^1,z^2,\pi)\in C^{\infty}(Q_0)^3$ is a solution of
\begin{equation}
\left\{\begin{aligned}
-2\partial_{x_3}\overline y^1 \partial_{x_1}z^1-\partial_{x_3}\overline y^2\partial_{x_2}z^1+(\partial_{x_1}\overline y^1-\partial_{x_3}\overline y^3)\partial_{x_3}z^1-\overline y^1 \partial^2_{x_1x_3}z^1\\-\overline y^2\partial^2_{x_2x_3}z^1
-\overline y^3\partial^2_{x_3x_3}z^1-\partial^2_{x_3t}z^1-\Delta \partial_{x_3}z^1
-\partial_{x_3}\overline y^2\partial_{x_1}z^2 \\+\partial_{x_1} \overline y^2\partial_{x_3}z^2=\partial_{x_3}\varphi^1-\partial_{x_1}\varphi^3,
\\
-\partial_{x_3}\overline y^1\partial_{x_2}z^1+\partial_{x_2}\overline y^1 \partial_{x_3}z^1-\partial_{x_3}\overline y^1\partial_{x_1}z^2-\overline y^1\partial^2_{x_1x_3}z^2-2\partial_{x_3}\overline y^2\partial_{x_2} z^2\\-\overline y^2 \partial^2_{x_2x_3}z^2+(\partial_{x_2}\overline y^2-\partial_{x_3}\overline y^3)\partial_{x_3}z^2
-\overline y^3\partial^2_{x_3x_3}z^2-\partial^2_{x_3t}z^2-\Delta \partial_{x_3}z^2\\=\partial_{x_3}\varphi^2-\partial_{x_2}\varphi^3,
\\
-\partial_{x_1}z^1-\partial_{x_2}z^2 =\varphi^4.
\label{sketch4}
\end{aligned}\right .
\end{equation}
then $(-\partial_{x_1}z^1,-\partial_{x_2}z^1,-\partial_{x_3}z^1,-\partial_{x_1}z^2,-\partial_{x_2}z^2,-\partial_{x_3}z^2)=\mathcal{N} \varphi$.
\item In Paragraph~\ref{denomb} we study the overdetermined system \eqref{sketch4}. If we consider $z^1$, $z^2$, and all their derivatives at every order as \emph{independent algebraic unknowns} (i.e. we forget that $\partial_{x_1}z^1,\ldots$ are derivatives of $z^1$ and consider them as unknowns of System \eqref{sketch4}), we obtain a system of $3$ equations with $20$ unknowns. However, we can prove that if we differentiate the equations of System \eqref{sketch4} enough times, one can obtain more equations than unknowns. In particular, if the two first equations of \eqref{sketch4} are differentiated 19 times and if the last equation of \eqref{sketch4} is differentiated 21 times, then we get $30360$ equations and $29900$ unknowns. We can write the big system describing these equations as follows:

$$L_0(t,x)Z= \Phi,$$ where $L_0\in C^\infty(Q_0;\mathcal{M}_{30360\times 29900}(\R))$, $Z\in \R^{29900}$ contains the derivatives of $z^1$ and $z^2$ up to the order $22$ and $\Phi \in \R^{30360}$ contains the derivatives of $\varphi$ up to the order $21$. If we are able to find a suitable submatrix of $L_0$ (denoted $P$) that is invertible, then roughly the matrix $P^{-1}$ (seen as a differential operator) will be a good candidate for $\mathcal N$.
\item In Paragraph~\ref{subsub-over}, we describe how we created a program that enables us to differentiate the equations of system \eqref{sketch4} and that finds a proper matrix $P$. Of course it cannot be done by hand, we have to use a computer. Let us point out that in the computer part of the proof, we only use \emph{symbolic computations}, so that \emph{no approximations} are made by the computer.

We first prove  (cf. Lemma~\ref{lemmalowerP}) that it is enough to find a suitable matrix $P$ which is invertible at some precise point $\xi^0$, i.e. it is enough to consider $L_0(\xi^0)$ for some well-chosen $\xi^0$. We explain in Lemma~\ref{lemm-good-reordering} how we found $P(\xi^0)$ thanks to a suitable reordering of matrix $L_0(\xi^0)$ given by the Dulmage-Mendelsohn decomposition of $L_0(\xi^0)$.
\end{enumerate}

\item We now remark that $(y^*+\tilde y,p^*+\tilde p,\tilde v)$ is a trajectory of \eqref{sketch1} (see \eqref{sketch2} and \eqref{sketch3}) that brings the initial condition $y^0$ to $0$ at time $T$. We then prove in Subsection~\ref{recolle} that $(y^*+\tilde y,p^*+\tilde p, \tilde v)$ is in some appropriate weighted Sobolev space (Proposition~\ref{cor-une-seule-composante}).
\item To conclude, in Section~\ref{IL}, we explain how the suitable functional setting we obtained for the solutions $(y,p,v)$ of System \eqref{sketch1} enables us to go back to the local null controllability of \eqref{NS} thanks to a usual argument of inverse mapping theorem.
\end{itemize}

\section{Constructing a relevant trajectory}
\label{secconstrbar}
In this subsection, we construct explicit particular trajectories $(\overline y,\overline p,\overline u)$
 going from $0$ to $0$ so that, as it will be shown in section \ref{seccontrollinearized}, the linearized control system around them
 is null controllable.

Without loss of generality we may assume that  $0\in \omega$. Let $g\in C^\infty (\R^3)$, $$g:(t,w,x_3)\mapsto g(t,w,x_3),$$
and $h\in C^\infty (\R^3)$, $$h:(t,w,x_3)\mapsto h(t,w,x_3).$$ For $(x_1,x_2,x_3)\in \R^3$, let $r:=\sqrt{x_1^2+x_2^2}$. We define $\overline y \in C^\infty(\R^4;\R^3)$ by
\begin{gather}
\label{defbary}
\overline y (t,x):=
\begin{pmatrix}
g(t,r^2,x_3)x_1\\
g(t,r^2,x_3)x_2\\
h(t,r^2,x_3)
\end{pmatrix}
, \forall t \in \R, \, \forall  x=(x_1,x_2,x_3)\in \R^3.
\end{gather}
 Let $r_1>0$
be small enough so that
\begin{gather}
\label{defC1+propC1}
\mathcal C_1:=\{(x_1,x_2,x_3)\in \R^3; r\leqslant r_1, |x_3|\leqslant r_1\}\subset \omega.
\end{gather}
On the functions $g$ and $h$, we also require that
\begin{gather}
\label{propsupportg}
\text{Supp}(g) \subset [T/4,T]\times (-\infty,r_1^2]\times[-r_1,r_1],
\\
\label{propsupporth}
\text{Supp}(h) \subset [T/4,T]\times(-\infty,r_1^2]\times[-r_1,r_1].
\end{gather}
In \eqref{propsupportg}, \eqref{propsupporth} and in the following, $\text{Supp}(f)$ denotes the support of the function $f$. From \eqref{defbary}, \eqref{defC1+propC1}, \eqref{propsupportg} and \eqref{propsupporth}, one obtains
\begin{gather}
\label{supportbary}
\text{Supp}(\overline y)
\subset [T/4,T]\times \mathcal C_1 \subset (0,T]\times \omega \subset (0,T]\times \Omega,
\end{gather}
which implies in particular that $\overline y$ has null trace on $\Sigma$.
Let $\widehat p \in  C^\infty (\R^3)$  be defined by
\begin{multline}
\label{eqdeftildepNS3}
\widehat p(t,w,x_3):= \frac{1}{2}\int_{w}^{r_1^2} \big(\partial_t g -(4w'\partial^2_{ww} g+8\partial_wg+\partial^2_{x_3x_3}g)
+2w'g\partial_wg+ g^2 \\+h\partial_{x_3}g\big)(t,w',x_3)dw'.
\end{multline}
%In \eqref{eqdeftildepNS3}, in order to lighten the notations,
% for $l\in \mathbb{N}^*$,
%$N \in \mathbb{N}^*$ and $i=(i_1,\ldots,i_N)$, with  $i_k\in\{t,w,3\}$ for every $k\in \{1,\ldots N\}$,
%and $\varphi :D(\varphi)\subset \R^3 \rightarrow \R^l$, $(t,w,x_3)\mapsto \varphi(t,w,x_3)$, $\varphi_{i_1i_2\ldots i_N}$ denotes the following derivative:
%$$
%\varphi_{i_1i_2\ldots i_N}:=\frac{\partial^N \varphi}{\partial \xi_{i_1}\ldots \partial \xi_{i_N}},
%$$
%where
%$$
%\xi_j= x_3 \text{ if } j=3, \, \xi_j= t \text{ if } j=t \text{ and } \xi_j= w \text{ if } j=w.
%$$
%We use similar notations throughout the paper.
Let $ \overline p  \in C^\infty (\R^4)$ be defined by
\begin{gather}
\label{eqdefbarpNS3}
\overline p(t,x_1,x_2,x_3):= \widehat p(t,r^2,x_3).
\end{gather}
 From  \eqref{defC1+propC1}, \eqref{propsupportg}, \eqref{propsupporth}, \eqref{eqdeftildepNS3} and  \eqref{eqdefbarpNS3}, it follows that
\begin{gather}
\label{supportbarpNS3D}
\text{Supp}(\overline p) \subset [T/4,T]\times \mathcal C_1 \subset [T/4,T]\times \omega \subset (0,T]\times \Omega.
\end{gather}
 From \eqref{defbary}, \eqref{eqdeftildepNS3} and  \eqref{eqdefbarpNS3}, one obtains
\begin{gather}
\label{eq1baryNS3D}
\overline y^1_t-\Delta \overline y^1 + (\overline y\cdot \nabla)\overline y^1+ \partial_{x_1}\overline p=0,
\\
\label{eq2baryNS3D}
\overline y^2_t-\Delta \overline y^2 + (\overline y\cdot \nabla)\overline y^2+  \partial_{x_2}\overline p=0.
\end{gather}
Let $\overline u \in C^\infty(\R^4)^3$ be defined by
\begin{gather}
\label{defbaruND3D}
\overline u :=(0,0,\overline y^3_t-\Delta \overline y^3+(\overline y\cdot \nabla )\overline y^3+ \partial_{x_3}\overline p).
\end{gather}
 From \eqref{defbaruND3D}, one obtains \eqref{cn}. From \eqref{supportbary}, \eqref{supportbarpNS3D} and \eqref{defbaruND3D}, we have

\begin{gather}
\label{supportbaru}
\text{Supp}(\overline u) \subset (0,T]\times \omega.
\end{gather}
 From \eqref{eq1baryNS3D}, \eqref{eq2baryNS3D}, \eqref{defbaruND3D} and \eqref{supportbaru}, we have
\begin{gather}
\label{barybarusolution}
\overline y_t-\Delta \overline y + (\overline y\cdot \nabla) \overline y+ \nabla \overline p=1_\omega \overline u.
\end{gather}
Finally, in order to have
\begin{gather}
\label{divbary=0}
\text{div }\overline y=0,
\end{gather}
it suffices to impose
\begin{gather}
\label{h3=}
\partial_{x_3} h=-2(g+w\partial_w g).
\end{gather}
Let $\nu$ be a positive numerical constant which will be chosen later. Let $a\in C^\infty(\R)$, $b\in C^\infty(\R)$ and $c\in C^\infty(\R)$ be such that
\begin{gather}
\label{propa}
\text{Supp}(a) \subset [T/4,T] \text{ and } a(t)=e^{\frac{-\nu}{(T-t)^5}} \text{ in } [T/2,T],
\\
\label{propb}
\text{Supp}(b) \subset (-\infty,r_1^2) \text{ and } b(w)= w, \, \forall s\in (-\infty,r_1^2/4],
\\
\label{propc}
\text{Supp}(c) \subset (-r_1,r_1) \text{ and } c(x_3) = x_3^2 \text{ in } [-r_1/2,r_1/2].
\end{gather}
We then set
\begin{gather}
\label{g=eabc}
g(t,w,x_3)=\varepsilon a(t)b(w)c'(x_3)
\end{gather}
and
\begin{gather}
\label{h=eabc}
h(t,w,x_3)=-2\varepsilon a(t)(b(w)+wb'(w))c(x_3),
\end{gather}
where $\epsilon>0$ (which will be chosen small enough later). From \eqref{g=eabc} and \eqref{h=eabc}, one obtains \eqref{h3=}.

In the next section, we prove that, for every small enough $T$, for every small enough $\varepsilon >0$ and for a well-chosen $\nu$, the linearized control system around
the trajectory $(\overline y,\overline p,\overline u)$ is controllable.

\section{A controllability result on the linearized system}
\label{seccontrollinearized}
\subsection{Definitions and notations}
\label{secmotivation-notations}
The linearized control system around the trajectory $(\overline y, \overline p, \overline u)$ is the linear control system
\begin{equation}
\label{linearizedbary}
\left\{\begin{aligned}
y^1_t-\Delta y^1+(\overline y\cdot \nabla)y^1 +(y\cdot \nabla){\overline y}^1+ \partial_{x_1}p&=f^1&\mbox{ in } Q,
\\
y^2_t-\Delta y^2+(\overline y\cdot \nabla)y^2 +(y\cdot \nabla){\overline y}^2+ \partial_{x_2}p&=f^2&\mbox{ in } Q,
\\
y^3_t-\Delta y^3+(\overline y\cdot \nabla)y^3 +(y\cdot \nabla){\overline y}^3+ \partial_{x_3}p&=1_\omega v+f^3&\mbox{ in } Q,
\\
\nabla \cdot y&=0&\mbox{ in }Q,
\\
y&= 0&\mbox{ on }\Sigma,
\end{aligned}\right .
\end{equation}
where the state is $y:Q\rightarrow \R^3$, $f:Q\rightarrow \R^3$ is a source term (it will be specified later in which space exactly it shall be) and the control is $v:Q\rightarrow \R$. 
In all what follows, in order to lighten the notations, we will write $\overline y$ as a function of $t$ and $x$ only, but one has to remember that $\overline y$ also depends on $\varepsilon$ and $\nu$.
Let $\omega_0$ be a nonempty open subset of
\begin{gather}
\label{defomega1}
\mathcal C_2:=\left\{(x_1,x_2,x_3); r<\frac{r_1}{2}, |x_3|<\frac{r_1}{2}\right\},
\end{gather}
which will be chosen more precisely in the next section.
Let $Q_0:=(T/2,T)\times \omega_0.$
The following figure summarizes the different roles of each open subset of $\Omega$ we introduced up to now.

\begin{figure}[!ht]
\begin{center}
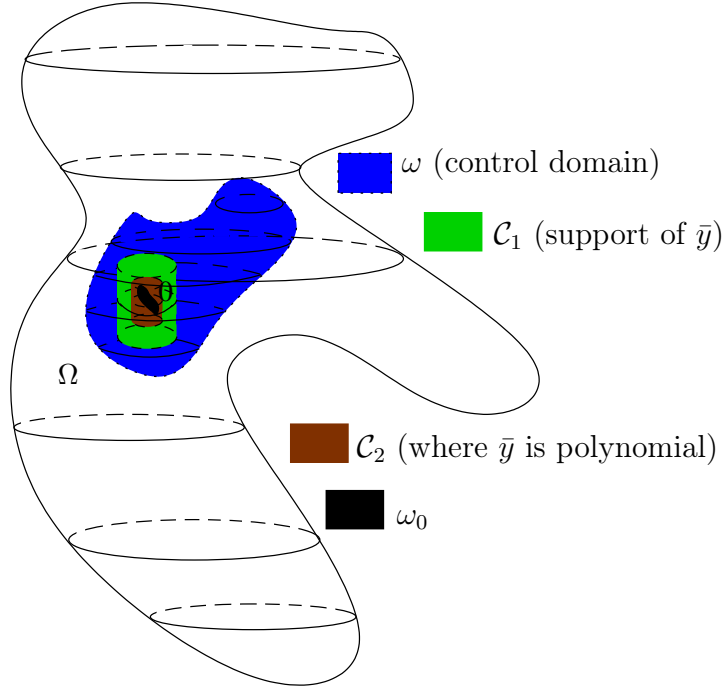
\caption{The open subsets $\mathcal C_1,\mathcal C_2,\omega_0,\omega$.}
\end{center}
\end{figure}

Let $\mathcal{L}: C^{\infty}(Q_0)^5\rightarrow C^{\infty}(Q_0)^4$ be defined by
\begin{gather}
\label{defmathcalL}
\mathcal{L}
\begin{pmatrix}
y
\\
p
\\
v
\end{pmatrix}
:=
\begin{pmatrix}
y^1_t-\Delta y^1+(\overline y\cdot \nabla)y^1 +(y\cdot \nabla){\overline y}^1+ \partial_{x_1}p
\\
y^2_t-\Delta y^2+(\overline y\cdot \nabla)y^2 +(y\cdot \nabla){\overline y}^2+ \partial_{x_2}p
\\y^3_t-\Delta y^3+(\overline y\cdot \nabla)y^3 +(y\cdot \nabla){\overline y}^3+ \partial_{x_3}p-v
\\
\nabla \cdot y
\end{pmatrix}
,\end{gather}
for every $y=(y^1,y^2,y^3)\in C^{\infty}(Q_0)^3$, for every $p \in C^{\infty}(Q_0) $
and for every $v\in C^{\infty}(Q_0)$. Let us denote by $$\xi:=(x_0,x_1,x_2,x_3)=(t,x_1,x_2,x_3)=(t,x)$$ the current point in $Q_0$.
For $\alpha =(\alpha_0, \alpha_1,\alpha_2,\alpha_3)\in \mathbb{N}^4$
and $\varphi :Q_0\rightarrow \R^k$, $\partial^\alpha \varphi$, denotes,
as usual, $$\partial^{\alpha_0}_{t^{\alpha_0}}\partial^{\alpha_1}_{x_1^{\alpha_1}}
\partial^{\alpha_2}_{x_2^{\alpha_2}}\partial^{\alpha_3}_{x_3^{\alpha_3}}\varphi.$$
%(in this part, this notation will be more convenient than the one introduced in the previous section, which is more compact only for derivatives of small order).
Let $\mathcal{L}(\R^k;\R^l)$ be the set of linear maps from $\R^k$ into $\R^l$ and $\mathcal{M}_{k,l}(\mathcal R)$ be the set of matrices of size $k\times l$ with values in the ring $\mathcal R$.

As usual, in the inequalities written in this article $C$ denotes a constant (depending in general only on $\omega$, $\Omega$, $T$) that may change from one line to another.

Let us give some other definitions.
\begin{Definition}
A linear map $\mathcal{M}: C^{\infty}(Q_0)^k\rightarrow C^{\infty}(Q_0)^l$ is called a \emph{linear partial differential operator of order $m$} if, for every
$\alpha = (\alpha_0, \alpha_1,\alpha_2,\alpha_3)\in \mathbb{N}^4$ with $|\alpha|:=\alpha_0+\alpha_1+\alpha_2+\alpha_3\leqslant m$, there exists
$A_\alpha \in C^\infty(Q_0; \mathcal{L}(\R^k;\R^l))$ such that
$$
(\mathcal{M}\varphi )(\xi)=\sum_{|\alpha|\leqslant m} A_\alpha (\xi)\partial^\alpha \varphi (\xi), \, \forall \xi \in  Q_0, \,
\forall \varphi \in C^{\infty}(Q_0)^k.
$$

A linear map $\mathcal{M}:C^{\infty}(Q_0)^k\rightarrow C^{\infty}(Q_0)^l$ is called a
\emph{linear partial differential operator} if there exists $m\in \mathbb{N}$ such that $\mathcal{M}$ is a linear partial differential operator of order $m$.
\end{Definition}

Let $k$ be a positive integer and
let
$\mathcal{B}:=(\mathcal{B}^1,\mathcal{B}^2,\mathcal{B}^3): C^{\infty}(Q_0)^k\rightarrow C^{\infty}(Q_0)^3$ be a linear
partial differential operator. Let us consider the linear equation
\begin{gather}
\label{equationypuf}
\mathcal{L}\begin{pmatrix}
y
\\
p
\\
v
\end{pmatrix}
=
\begin{pmatrix}
\mathcal{B}^1 u
\\
\mathcal{B}^2 u
\\
\mathcal{B}^3 u
\\
0
\end{pmatrix}
,
\end{gather}
where the data is $u\in C^{\infty}(Q_0)^k$ and the unknown is $(y,p,v)\in C^{\infty}(Q_0)^5$.
Following \cite[p.~148]{Gromovbook}, we adopt the following definition.
\begin{Definition}
\label{defialgebraicsolv}
The linear equation \eqref{equationypuf}
is \emph{algebraically solvable} if there  exists a linear
partial differential operator $\mathcal{M}:C^{\infty}(Q_0)^k\rightarrow C^{\infty}(Q_0)^5$ such that, for every $u\in C^{\infty}(Q_0)^k$, $\mathcal{M}u$ is a solution of \eqref{equationypuf}, i.e. such that
\begin{gather}
\label{LMB}
\mathcal L \circ \mathcal M=(\mathcal B,0).
\end{gather}
\end{Definition}
In the following, every function $\varphi \in C^{\infty}(Q_0)^l$ with a compact support included in $Q_0$ is extended by $0$ in $Q\setminus Q_0$ and we still denote this extension by $\varphi$.

The next proposition explains how the notion of  ``algebraic solvability'' can be useful to reduce the number of controls as soon as a controllability result is already known for a large number of controls. In fact, the question of the null-controllability of \eqref{linearizedbary} can be split up into two distinct problems: One ``algebraic'' part (solving system\eqref{LMB}) and one ``analytic'' part (finding controls which are in the image of $\mathcal B$, the control acting possibly on all the equations and not only on the third one). This proposition has a very general scope and could be formulated for more general control systems. It is inspired by techniques used in the control of ordinary differential equations (see, in particular, \cite[Chapter 1, pages 13-15]{MR2302744}).
\begin{Proposition}
\label{prop-alg-controllability}
Let us consider the linear control system
\begin{gather}
\label{linearizedbary-Bf-Cauchy}
\left\{\begin{aligned}
y^1_t-\Delta y^1+(\overline y\cdot \nabla)y^1 +(y\cdot \nabla){\overline y}^1+ \partial_{x_1}p&=\mathcal{B}^1u+f^1&\mbox{ in } Q,
\\
y^2_t-\Delta y^2+(\overline y\cdot \nabla)y^2 +(y\cdot \nabla){\overline y}^2+ \partial_{x_2}p&=\mathcal{B}^2u+f^2&\mbox{ in } Q,
\\
y^3_t-\Delta y^3+(\overline y\cdot \nabla)y^3 +(y\cdot \nabla){\overline y}^3+ \partial_{x_3}p&=\mathcal{B}^3u+f^3&\mbox{ in } Q,
\\
\nabla \cdot y&=0&\mbox{ in }Q,
\\
y&= 0&\mbox{ on }\Sigma,
\\
y(0,\cdot)&=y^0 &\mbox{ in } \Omega,
\end{aligned}\right .
\end{gather}
where the state is $y:Q\rightarrow \R^3$, the control is $u\in C^{\infty}(Q)^k$, which is
required to have a support in $Q_0$, and $f:=(f^1,f^2,f^3)\in C^\infty(Q)^3$ is a source term. Let us assume that:
\begin{itemize}
 \item[\hypertarget{assu1}{$\mathcal{A}_1$.}]
 The linear control system \eqref{linearizedbary-Bf-Cauchy} is null controllable during the interval of time
$[0,T]$ in the sense that for every
$y^0\in V$ and  for every $f\in C^\infty(Q)$ such that
\begin{gather}
\label{vanispresT}
\text{there exists $\delta>0$ such that $f=0$ on $[T-\delta,T]\times \Omega$},
\end{gather}
there exists $u\in C^{\infty}(Q)^k$ with a compact support included in $Q_0$
such that the solution $(y,p)$ of \eqref{linearizedbary-Bf-Cauchy}  with initial
condition $y(0,\cdot)=y^0$ satisfies $y(T,\cdot)=0$.
\item [\hypertarget{assu2}{$\mathcal{A}_2.$}]
\begin{gather*}
\text{\eqref{equationypuf} is algebraically solvable.}
\end{gather*}
\end{itemize}
Then, the linear control system \eqref{linearizedbary}
is null controllable during the interval of time
  $[0,T]$: For every
$y^0\in V$ and for every $f\in C^\infty(Q)$ satisfying \eqref{vanispresT}, there exists  $v\in C^{\infty}(Q)$ with a compact support included in $Q_0$ such that
the the solution $(y,p)$ of \eqref{linearizedbary} with 
initial condition $y(0,\cdot)=y^0$ satisfies $y(T,\cdot)=0$.
\end{Proposition}

\textbf{Proof of Proposition~\ref{prop-alg-controllability}.}

First of all, we use the null-controllability of \eqref{linearizedbary-Bf-Cauchy} with controls in the image of $\mathcal B$ (Assumption \hyperlink{assu1}{$\mathcal{A}_1$}) with source term $f$:
 let $y^0\in V$  and let
 $u^*\in C^{\infty}(Q_0)^k$ with compact supported included in some open subset $Q^*\subset\subset Q_0$ included in $Q_0$ such that the solution $(y^*,p^*)$ of the following equation:
\begin{gather}
\label{linearizedbary-Bf-Cauchy-*}
\left\{\begin{aligned}
y^{*1}_t-\Delta y^{*1}+(\overline y\cdot \nabla)y^{*1} +(y^*\cdot \nabla){\overline y}^1+ \partial_{x_1}p&=\mathcal{B}^1u^*+f^1&\mbox{ in } Q,
\\
y^{*2}_t-\Delta y^{*2}+(\overline y\cdot \nabla)y^{*2} +(y^*\cdot \nabla){\overline y}^2+ \partial_{x_2}p&=\mathcal{B}^2u^*+f^2&\mbox{ in } Q,
\\
y^{*3}_t-\Delta y^{*3}+(\overline y\cdot \nabla)y^{*3} +(y^*\cdot \nabla){\overline y}^3+ \partial_{x_3}p&=\mathcal{B}^3u^*+f^3&\mbox{ in } Q,
\\
\nabla \cdot y^*&=0&\mbox{ in }Q,
\\
y^*&= 0&\mbox{ on }\Sigma,
\end{aligned}\right .
\end{gather}
with initial condition $y^*(0,\cdot)=y^0$ satisfies $y^*(T,\cdot)=0$. Let us remark that $\mathcal B$ is a local operator, which implies that $\mathcal{B}u^*_{|Q_0}$ still has a compact support included in $Q^*$. Now we use the algebraic solvability of \eqref{equationypuf} (Assumption \hyperlink{assu2}{$\mathcal{A}_2$}). Let $\mathcal{M} $ be as in Definition~\ref{defialgebraicsolv}. For a map $h\in C^\infty(Q)^k$ with a support included in $Q_0$, we denote by $\mathcal{M}$ the map from $Q$ into $\R^5$ defined by
 \begin{gather*}
\mathcal{M}h= 0 \text{ in } Q\setminus Q_0, \mathcal{M}h=\mathcal{M}(h_{|Q_0}) \text{ in } Q_0.
\end{gather*}
We shall use this slight abuse of notation until the end of the paper. Note that, for every $h\in C^\infty(Q)^k$
with a support included in $Q_0$, $\mathcal{M}h \in C^\infty(Q)^5$ and has a support included in $Q_0$ (because $\mathcal M$ is a local operator). Let us call
$$(\tilde y,\tilde p, \tilde v):=-\mathcal{M}u^*,$$
so that $(\tilde y,\tilde p,\tilde v)$ verifies the following linearized Navier-Stokes equation:
\begin{gather}
\label{algebr1}
\left\{\begin{aligned}
\tilde y^1_t-\Delta \tilde y^1+(\overline y\cdot \nabla)\tilde y^1 +(\tilde y\cdot \nabla){\overline y}^1+
 \partial_{x_1}\tilde p&=-\mathcal{B}^1u^*&\mbox{ in } Q,
\\
\tilde y^2_t-\Delta \tilde y^2+(\overline y\cdot \nabla)\tilde y^2 +(\tilde y\cdot \nabla){\overline y}^2+ \partial_{x_2}\tilde p&=-\mathcal{B}^2u^*&\mbox{ in } Q,
\\
\tilde y^3_t-\Delta \tilde y^3+(\overline y\cdot \nabla)\tilde y^3 +(\tilde y\cdot \nabla){\overline y}^3+ \partial_{x_3}\tilde p&=\tilde v -\mathcal{B}^3u^*&\mbox{ in } Q,
\\
\nabla \cdot \tilde y&=0&\mbox{ in }Q,
\\
\tilde y&= 0&\mbox{ on }\Sigma.
\\
\end{aligned}\right .
\end{gather}
%%%achtung
 One observes that the support of $(\tilde y,\tilde p, \tilde v)$ is still included in $Q^*$ (which is strongly included in $Q_0$). In particular $\tilde{y}(0,\cdot)=0$ and $\tilde{y}(T,\cdot )= 0$. Let $$(y,p,v):=(y^*+\tilde y,p^*+\tilde p,\tilde v).$$ Note that $(y,p)$ is different from $(y^*,p^*)$ only  on $Q^*$. In particular one has $y(0,\cdot)=y^0$ and $y(T,\cdot)=0$. Moreover, from \eqref{linearizedbary-Bf-Cauchy-*} and \eqref{algebr1}, we obtain that $(y,p,v)$ verifies the equation
\begin{gather*}
\left\{\begin{aligned}
y^1_t-\Delta y^1+(\overline y\cdot \nabla)y^1 +(y\cdot \nabla){\overline y}^1+ \partial_{x_1}p&=f^1&\mbox{ in } Q,
\\
y^2_t-\Delta y^2+(\overline y\cdot \nabla)y^2 +(y\cdot \nabla){\overline y}^2+ \partial_{x_2}p&=f^2&\mbox{ in } Q,
\\
y^3_t-\Delta y^3+(\overline y\cdot \nabla)y^3 +(y\cdot \nabla){\overline y}^3+ \partial_{x_3}p&= 1_{\omega^*} v +f^3&\mbox{ in } Q,
\\
\nabla \cdot y&=0&\mbox{ in }Q,
\\
y&= 0&\mbox{ on }\Sigma,
\end{aligned}\right.
\end{gather*}
 which shows that the linear control system \eqref{linearizedbary} is indeed null controllable during the interval of time
$[0,T]$ and concludes
the proof of Proposition~\ref{prop-alg-controllability}. \cqfd
\begin{Remark}
\label{not-so-regular}For the sake of simplicity, we have formulated Proposition~\ref{prop-alg-controllability} in a $C^\infty$ setting. Let us assume that the control $u$ coming from Assumption \hyperlink{assu1}{$\mathcal{A}_1$} is not of class $C^\infty$, but is less regular (one sees that the regularities of $y^*$, $p^*$ and $f$ does not matter for the proof of Proposition~\ref{prop-alg-controllability} since only $u^*$ is differentiated by the linear partial differential operator $\mathcal M$). For example, assume that $u^*\in H_1$ where $H_1$ is a functional space (for example a weighted Sobolev space), and assume that $\mathcal M$ can be extended on $H_1$, $\mathcal M u^*$ being then in another functional space $H_2$ (for example another weighted Sobolev space of order less that $H_1$ in order to take into account that $\mathcal M$ is a linear partial differential operator). Then one easily verifies that Proposition~\ref{prop-alg-controllability} remains true as soon as every function of $H_1$ (and its derivatives until the order at least the order of $\mathcal M$) vanishes at time $t=T$, the first Assumption \hyperlink{assu1}{$\mathcal{A}_1$} being changed as the following:
the linear control system \eqref{linearizedbary-Bf-Cauchy} is null controllable during the interval of time
  $[0,T]$, i.e. for every
$y^0\in V$ and for every $f\in L^2(Q)$ satisfying \eqref{vanispresT} there exists
$u\in H_1$ with support included in $Q_0$
such that the solution $(y,p)$ of \eqref{linearizedbary-Bf-Cauchy} satisfying the
initial condition $y(0,\cdot)=y^0$ satisfies $y(T,\cdot)=0$. Note that the scalar control $v$ is now only in $H_2$. Similarly, we will need to relax property \eqref{vanispresT} by replacing it with a suitable decay rate near $t=T$.
This will be detailed in Subsection~\ref{recolle-morceaux}.
\end{Remark}

It remains to deal, for a suitable choice of $\mathcal{B}$, with Assumption \hyperlink{assu2}{$\mathcal{A}_2$} (we shall
do it in Subsection~\ref{secalgebraicsolvability}) and with Assumption \hyperlink{assu1}{$\mathcal{A}_1$}, i.e. with the null controllability
of the linear control system \eqref{linearizedbary-Bf-Cauchy} in suitable spaces
(we shall do it in Subsection~\ref{seccontrollabilitywithB}).

\subsection{Algebraic solvability of \texorpdfstring{\eqref{equationypuf}}{}}
\label{secalgebraicsolvability} We choose $k=7$ and define $\mathcal{B}$ by
\begin{gather}
\label{defmathcalB}
\mathcal{B}(f^1,f^2,f^3,f^4,f^5,f^6,f^7):=
\begin{pmatrix}
 \partial_{x_1}f^1+ \partial_{x_2}f^2+ \partial_{x_3}f^3
\\
 \partial_{x_1}f^4+ \partial_{x_2}f^5+ \partial_{x_3}f^6
\\
f^7
\end{pmatrix}
.
\end{gather}
The main result of this subsection is the following proposition.
\begin{Proposition}
\label{prop-solvable}There exists $\varepsilon^*>0$, there exists $T^*>0$ such that, for every $\varepsilon\in (0,\varepsilon^*)$, there exists a nonempty
open subset $\omega_0$ of $\mathcal C_2$ such that Assumption \hyperlink{assu2}{$\mathcal{A}_2$} holds for every $T<T^*$: There exists a
linear partial differential operator $\mathcal{M}:C^{\infty}(Q_0)^7\rightarrow C^{\infty}(Q_0)^5$ such that \eqref{LMB} holds.
\end{Proposition}
\subsubsection{The adjoint problem}
\label{subsub-adj}
Let $\mathcal{L}_0: C^{\infty}(Q_0)^4\rightarrow C^{\infty}(Q_0)^3$
be the linear partial differential operator defined by
\begin{gather}
\label{defmathcalL0}
\mathcal{L}_0
\begin{pmatrix}
y
\\
p
\end{pmatrix}
:=
\begin{pmatrix}
y^1_t-\Delta y^1+(\overline y\cdot \nabla)y^1 +(y\cdot \nabla){\overline y}^1+ \partial_{x_1}p
\\
y^2_t-\Delta y^2+(\overline y\cdot \nabla)y^2 +(y\cdot \nabla){\overline y}^2+ \partial_{x_2}p
\\
\nabla \cdot y
\end{pmatrix}
,
\end{gather}
for every $y=(y^1,y^2,y^3)\in C^{\infty}(Q_0)^3$, and every $p\in C^{\infty}(Q_0)$.

Let $\mathcal{B}_0:C^{\infty}(Q_0)^6\rightarrow C^{\infty}(Q_0)^3$ be
the linear partial differential operator defined by
\begin{gather}
\label{deftildeB}
\
\mathcal{B}_0(f^1,f^2,f^3,f^4,f^5,f^6):=
\begin{pmatrix}
 \partial_{x_1}f^1+ \partial_{x_2}f^2+ \partial_{x_3}f^3
\\
 \partial_{x_1}f^4+ \partial_{x_2}f^5+ \partial_{x_3}f^6
\\
0
\end{pmatrix}
.
\end{gather}
Note that the third equation of \eqref{equationypuf} can be read as
$$
v=y^3_t-\Delta y^3+(\overline y\cdot \nabla)y^3 +(y\cdot \nabla){\overline y}^3+\partial_3p-f^7.
$$
Hence, one easily sees that Assumption \hyperlink{assu2}{$\mathcal{A}_2$}
 is equivalent to the existence
of a linear partial differential operator $\mathcal{M}_0: C^{\infty}(Q_0)^6\rightarrow C^{\infty}(Q_0)^4$ such that
\begin{gather}
\label{eqL0M0}
\mathcal{L}_0\circ \mathcal{M}_0=\mathcal{B}_0.
\end{gather}

As in \cite[p.~157]{Gromovbook}, we study \eqref{eqL0M0} by looking at the ``adjoint equation''.
For every linear partial differential operator
$\mathcal{M}: C^{\infty}(Q_0)^k\rightarrow C^{\infty}(Q_0)^l$, $\mathcal{M}=\sum_{|\alpha|\leqslant m} A_\alpha \partial^\alpha$, we associate its (formal) adjoint
$$\mathcal{M}^*: C^{\infty}(Q_0)^l\rightarrow C^{\infty}(Q_0)^k$$
 defined
 by
\begin{gather}
\label{defadjointformel}
\mathcal{M}^* \psi := \sum_{|\alpha|\leqslant m}(-1)^{|\alpha|} \partial^\alpha (A_\alpha ^T \psi), \, \forall \psi \in C^{\infty}(Q_0)^l,
\end{gather}
where $A_\alpha^T (\xi)$ is the transpose of the matrix $A_\alpha(\xi)$. (Definition \eqref{defadjointformel} makes sense since $$\sum_{|\alpha|\leqslant m} A_\alpha \partial^\alpha=0$$ implies that the
$A_\alpha$ are all equal to $0$.)
One has $\mathcal{M}^{**}=\mathcal{M}$ and, if $\mathcal{M}: C^{\infty}(Q_0)^k\rightarrow C^{\infty}(Q_0)^l$ and
$\mathcal{N}: C^{\infty}(Q_0)^l\rightarrow C^{\infty}(Q_0)^m$ are two linear partial differential operators, then
$ (\mathcal{N} \circ \mathcal{M} )^*=\mathcal{M}^*\circ \mathcal{N}^*$.

Hence, \eqref{eqL0M0} is equivalent to
\begin{gather}
\label{eqM0adjL0adj}
\mathcal{M}_0^*\circ \mathcal{L}_0^*=\mathcal{B}_0^*.
\end{gather}
Direct computations, together with \eqref{divbary=0}, show that, for every $z=(z^1,z^2)\in C^{\infty}(Q_0)^2$
and for every $\pi \in C^{\infty}(Q_0)$,
\begin{gather}
\label{calculLadjointL0}
\mathcal{L}^*_0
\begin{pmatrix}
z
\\
\pi
\end{pmatrix}
=
\begin{pmatrix}
-z^1_t-\Delta z^1-(\overline y\cdot \nabla)z^1+\partial_{x_1}\overline y^1 z^1+\partial_{x_1}\overline y^2 z^2-\partial_{x_1}\pi
\\
-z^2_t-\Delta z^2-(\overline y\cdot \nabla)z^2+\partial_{x_2}\overline y^1 z^1+\partial_{x_2}\overline y^2 z^2-\partial_{x_2}\pi
\\
\partial_{x_3}\overline y^1 z^1+\partial_{x_3}\overline y^2 z^2-\partial_{x_3}\pi
\\
-\partial_{x_1}z^1-\partial_{x_2}z^2
\end{pmatrix}
,
\\
\label{calculLadjointB0}
\mathcal{B}_0^*
\begin{pmatrix}
z
\\
\pi
\end{pmatrix}
=(-\partial_{x_1}z^1,-\partial_{x_2}z^1,-\partial_{x_3}z^1,-\partial_{x_1}z^2,-\partial_{x_2}z^2,-\partial_{x_3}z^2).
\end{gather}
Assumption \hyperlink{assu2}{$\mathcal{A}_2$} is now equivalent to the following property: There exists
a linear partial differential operator $\mathcal{N} (=\mathcal{M}_0^*):C^{\infty}(Q_0)^4
\rightarrow C^{\infty}(Q_0)^6$ such that
for every $\varphi=(\varphi^1,\varphi^2,\varphi^3,\varphi^4)\in C^{\infty}(Q_0)^4$, if
$(z^1,z^2,\pi)\in C^{\infty}(Q_0)^3$ is a solution of
\begin{gather}
\label{eqaresoudre}
\left\{
\begin{array}{l}
-z^1_t-\Delta z^1-(\overline y\cdot \nabla)z^1+\partial_{x_1}\overline y^1 z^1+\partial_{x_1}\overline y^2 z^2-\partial_{x_1}\pi= \varphi^1,
\\
-z^2_t-\Delta z^2-(\overline y\cdot \nabla)z^2+\partial_{x_2}\overline y^1 z^1+\partial_{x_2}\overline y^2 z^2-\partial_{x_2}\pi=\varphi^2,
\\
\partial_{x_3}\overline y^1 z^1+\partial_{x_3}\overline y^2 z^2-\partial_{x_3}\pi=\varphi^3,
\\
-\partial_{x_1}z^1-\partial_{x_2}z^2=\varphi^4,
\end{array}
\right.
\end{gather}
then $(-\partial_{x_1}z^1,-\partial_{x_2}z^1,-\partial_{x_3}z^1,-\partial_{x_1}z^2,-\partial_{x_2}z^2,-\partial_{x_3}z^2)=\mathcal{N} \varphi$.

\begin{Remark}
\label{remarkpasresoluble}
The most natural linear partial differential operator $\mathcal{B}$ to try first would have been
$\mathcal{B} : C^{\infty}(Q_0)^3\rightarrow C^{\infty}(Q_0)^3$ defined by
\begin{gather}
\label{naturalB}
\mathcal{B}f:=
\begin{pmatrix}
f^1
\\
f^2
\\
f^3
\end{pmatrix}
,\, \forall f =(f^1,f^2,f^3)\in C^{\infty}(Q_0)^3.
\end{gather}
Unfortunately, Proposition~\ref{prop-solvable} does not hold with this $\mathcal{B}$. Indeed, in this case
$\mathcal{B}_0^*: C^{\infty}(Q_0)^3\rightarrow C^{\infty}(Q_0)^2$ would be now (compare with
\eqref{calculLadjointB0}) such that, for every $z=(z^1,z^2)\in C^{\infty}(Q_0)^2$
and for every $\pi \in C^{\infty}(Q_0)$,
\begin{gather}
\label{calculLadjointB0-new}
\mathcal{B}_0^*(z,\pi)=(z^1,z^2).
\end{gather}
Let
$F_1\in C^\infty(T/2,T)$ and let $F_2\in C^\infty(T/2,T)$. We define $z=(z^1,z^2)
\in C^\infty(Q_0;\R^2)$ and $\pi \in C^{\infty}(Q_0)$ by
\begin{gather*}
z^1(t,x):=F_1(t), \\ z^2(t,x):=F_2(t),\\\pi(t,x):= -F'_1(t) x_1 -F'_2(t) x_2+F_1(t)\overline y^1+F_2(t)\overline y^2.
\end{gather*}
Then $\mathcal{L}_0^*(z,\pi)=0$. However, if $(F_1,F_2)\not = (0,0)$,
then $\mathcal{B}_0^*(z,\pi)\not =0$. Hence, in this case,
\eqref{eqM0adjL0adj} does not hold whatever the linear partial
differential operator $\mathcal{M}_0$ is and whatever the trajectory $(\overline y, \overline p,\overline u)$ is.
\end{Remark}

\subsubsection{Number of variables and equations}
\label{denomb}

Let us give some algebraic results about the number of derivatives of a certain order.
\begin{Definition}
Consider a scalar PDE with a smooth (enough) variable $z$ depending on $4$ variables $x_0,x_1,x_2,x_3$. We call \emph{equations of level} $n$ all the different equations we obtain by differentiating the PDE with respect to all the possible multi-integers of length $n$. The number of ``distinct'' equations of level $n$ is denoted $E(n)$, and the number of ``distinct'' equations of a level less than or equal to $n$ is denoted $F(n)$.
\end{Definition}
\begin{Remark}
Clearly, $E(n)$ is also the distinct number of derivatives of order $n$ for (smooth enough) functions having $4$ variables, and $F(n)$ is also the distinct number of derivatives of an order less or equal than $n$ for (smooth enough) functions having $4$ variables. Moreover,
if we consider a scalar PDE with many variables $z^1,\ldots,z^k$ depending on $x_0,x_1,x_2,x_3$ containing derivatives of $z^1,\ldots z^k$ of order $m$ at most, the maximum number of derivatives of $z^1,\ldots z^k$ we may expect in the equations of a level less than or equal to $n$ is $kF(n+m)$.
\end{Remark}
We want to compute $E$ and $F$ precisely. One has
\begin{gather}
\label{expEn}E(n)=\displaystyle\frac{(n+1)(n+2)(n+3)}{6},
\\
\label{expFn}F(n)=\displaystyle\frac{(n+1)(n+2)(n+3)(n+4)}{24}.
\end{gather}
Indeed
\begin{equation*}
(\alpha_0,\alpha_1,\alpha_2,\alpha_3)\mapsto
\{\alpha_0+1,\alpha_0+\alpha_1+2,\alpha_0+\alpha_1+\alpha_2+3,\alpha_0+\alpha_1+\alpha_2+\alpha_3+4\}
\end{equation*}
defines a bijection between the set of $(\alpha_0,\alpha_1,\alpha_2,\alpha_3)\in \mathbb{N}^4$ such that $\alpha_0+\alpha_1+\alpha_2+\alpha_3\leqslant n$ and the set of subsets of $\{1,2,\ldots,n+4\}$ having 4 elements.
Hence, $F(n)$ being the number of $(\alpha_0,\alpha_1,\alpha_2,\alpha_3)\in \mathbb{N}^4$ such that $\alpha_0+\alpha_1+\alpha_2+\alpha_3\leqslant n$, we have \eqref{expFn}. In order to obtain \eqref{expEn},
it suffices to notice that
\begin{gather*}
\begin{aligned} E(n)&=F(n)-F(n-1)
\\&=\displaystyle\frac{(n+1)(n+2)(n+3)}{6}.
\end{aligned}
\end{gather*}

\subsubsection{A related overdetermined system}
\label{subsub-over}
Let us now study the equation \eqref{eqaresoudre}, where the data
is $(\varphi^1,\varphi^2,\varphi^3,\varphi^4) \in C^{\infty}(Q_0)^4$
and the unknown is $(z^1,z^2,\pi)\in C^{\infty}(Q_0)^3$.

Let us explain the idea behind the reasoning we are going to develop in this subsection. Equation \eqref{eqaresoudre} is ``analytically'' overdetermined, since we have more equations ($4$) than unknowns ($3$). However, if we see \eqref{eqaresoudre} as a linear system of algebraic unknowns (the unknowns being $z^1,z^2,\pi$ and their derivatives) the system is now ``algebraically'' underdetermined: We have $4$ equations and $19$ unknowns. But it is easy to obtain as many new equations as we want: It suffices to differentiate \eqref{eqaresoudre} enough times. Some new ``algebraic unknowns'' (the derivatives of $z^1,z^2,\pi$) appear, but since the system was ``analytically'' overdetermined, one can hope that they are not ``too many'' new unknowns appearing. Notably, one can hope that, after differentiating a sufficient number of times, we obtain more equations than ``algebraic unknowns''. We would then deduce Assumption \hyperlink{assu2}{$\mathcal{A}_2$} by ``inverting'' in some sense this well-posed linear system (this will be explained in detail later).

We first eliminate $\pi$ in our equation \eqref{eqaresoudre}. To reach this goal, in \eqref{eqaresoudre},
we apply $\partial_3$ to the first and second lines, and use the third line. We obtain the following equations:
\begin{equation}
\left\{\begin{aligned}
-2\partial_{x_3}\overline y^1 \partial_{x_1}z^1-\partial_{x_3}\overline y^2\partial_{x_2}z^1+(\partial_{x_1}\overline y^1-\partial_{x_3}\overline y^3)\partial_{x_3}z^1-\overline y^1 \partial^2_{x_1x_3}z^1\\-\overline y^2\partial^2_{x_2x_3}z^1
-\overline y^3\partial^2_{x_3x_3}z^1-\partial^2_{x_3t}z^1-\Delta \partial_{x_3}z^1
-\partial_{x_3}\overline y^2\partial_{x_1}z^2 +\partial_{x_1} \overline y^2\partial_{x_3}z^2\\=\partial_{x_3}\varphi^1-\partial_{x_1}\varphi^3,
\\
-\partial_{x_3}\overline y^1\partial_{x_2}z^1+\partial_{x_2}\overline y^1 \partial_{x_3}z^1-\partial_{x_3}\overline y^1\partial_{x_1}z^2-\overline y^1\partial^2_{x_1x_3}z^2-2\partial_{x_3}\overline y^2\partial_{x_2} z^2\\-\overline y^2 \partial^2_{x_2x_3}z^2+(\partial_{x_2}\overline y^2-\partial_{x_3}\overline y^3)\partial_{x_3}z^2
-\overline y^3\partial^2_{x_3x_3}z^2-\partial^2_{x_3t}z^2-\Delta \partial_{x_3}z^2\\=\partial_{x_3}\varphi^2-\partial_{x_2}\varphi^3,
\\
-\partial_{x_1}z^1-\partial_{x_2}z^2 =\varphi^4.
\label{eqzphi}
\end{aligned}\right .
\end{equation}
The first and second equation of \eqref{eqzphi} contain derivatives
of $z^1$ and $z^2$
up to order $3$  and the third equation derivatives up to order  $1$. We would like to have the same maximal order of derivatives appearing in the three equations in order to be sure that the derivatives of maximal order appearing in the first and second equation might also appear in the third one. Hence we are going to differentiate the last equation $2$ more times than the others. If we count the maximum number of derivatives of $z^1$ and $z^2$ we create by differentiating $n$ times the first and second equation and $n+2$ times the third one, we obtain
\begin{gather}
\label{valuehn}
H(n)=2F(n+3)=\frac{(n+4)(n+5)(n+6)(n+7)}{12}
\end{gather}
different derivatives. The number $G(n)$ of equations we obtain is then
\begin{gather}
\begin{aligned}
G(n)&= 2F(n)+F(n+2)
\\&=\frac{(3+n)(4+n)(34+17n+3n^2)}{24}.
\label{valuegn}
\end{aligned}
\end{gather}
 From \eqref{valuehn} and \eqref{valuegn}, one sees that $G(n)-H(n)$ is increasing with respect to $n$
 and that
\begin{gather*}
G(18)-H(18)=-44<0 \text{ and } G(19)-H(19)=460>0.
\end{gather*}
Hence, in order to have more equations than unknowns and as few equations as possible, we choose $n=19$.
We have $G(19)=30360$ equations and $H(19)=29900$ unknowns. We can see this system of $30360$ partial differential equations as a linear system
$$L_0(t,x)Z= \Phi,$$ where $L_0\in C^\infty(Q_0;\mathcal{M}_{30360\times 29900}(\R))$, $Z\in \R^{29900}$ ($Z$ contains the derivatives of $z^1$ and $z^2$ up to the order $19$) and $\Phi \in \R^{30360}$ ($\Phi$ contains the derivatives of $\varphi$ up to the order $19$). Note that $L_0$ also depends on $\varepsilon$ and $\nu$, but this does not need to be emphasized in what follows. Hence, in order to lighten the notations, we will only see $L_0$ as a function of $t$ and $x$ (as for $\tilde{L_0},N,\ldots$ that are be introduced later).
We order the $29900$ lines of $Z$ so that
\begin{gather*}
Z^1=\partial_{x_1}z^1,\, Z^2=\partial_{x_2}z^1,\,
Z^3=\partial_{x_3}z^1,\, Z^4=\partial_{x_1}z^2,\,
Z^5=\partial_{x_2}z^2,\, Z^6=\partial_{x_3}z^2.
\end{gather*}
Assumption \hyperlink{assu2}{$\mathcal{A}_2$} can then be written as follows: Prove the existence of a nonempty
open subset $\omega_0$ of $\mathcal C_2$ and of a map $N\in C^\infty(Q_0; \mathcal{M}_{6\times 30360}(\R))$ ($N$ is the algebraic version of the linear partial differential operator $\mathcal N$ introduced in Subsection~\ref{secalgebraicsolvability}, every linear partial differential operator can be alternatively considered as a matrix acting on the derivatives of the input functions) such that
\begin{gather}
\label{NM}
N(t,x){L_0}(t,x)Z=
(Z^1,Z^2,Z^3,Z^4,Z^5,Z^6),
\, \forall (t,x) \in Q_0,\,  \forall Z \in \R^{29900}.
\end{gather}

Since the size of the matrix ${L_0}(t,x)$  is very large, it is impossible to find some $N$ verifying System \eqref{NM} by hand and we will have to do computations on a computer. Notably, it would be more convenient to make ${L_0}$ be
a sparse matrix in order to use relevant tools adapted to the study of big sparse linear systems. This is the reason for our
 simple choices for $a$, $b$ and $c$ given in \eqref{propa}, \eqref{propb} and \eqref{propc} (polynomials of small order do not create to many non zero  coefficients in ${L_0}$ when they are differentiated). Using
\eqref{defbary}, \eqref{propa}, \eqref{propb}, \eqref{propc}, \eqref{g=eabc} and \eqref{h=eabc}, System \eqref{eqzphi}  becomes simply, in $Q_0$,
\begin{equation}
\left\{\begin{aligned}
a(t)(-4x_1^3-4 x_1 x_2^2)\varepsilon \partial_{x_1}z^1+a(t)(-2 x_1^2 x_2-2 x_2^3)\varepsilon \partial_{x_2}z^1
+a(t)(14 x_1^2 x_3
\\
+10 x_2^2 x_3)\varepsilon
\partial_{x_3}z^1
+a(t)(-2
x_1^2 x_2-2 x_2^3)\varepsilon \partial_{x_1}z^2+4a(t) x_1 x_2 x_3\varepsilon
\partial_{x_3}z^2+a(t)\\(-2 x_1^3 x_3-2 x_1 x_2^2
x_3)\varepsilon \partial^2_{x_1x_3}z^1+a(t)(-2 x_1^2 x_2 x_3-2 x_2^3 x_3)
\varepsilon \partial^2_{x_2x_3}z^1+a(t)
\\(4 x_1^2 x_3^2+4 x_2^2 x_3^2)\varepsilon
\partial^2_{x_3x_3}z^1
-\partial^2_{x_3t}z^1-\partial^3_{x_1x_3x_3}z^1-\partial^3_{x_2x_2x_3}z^1-\partial^3_{x_3x_3x_3}z^1_{333}\\
=\partial_{x_3}\varphi^1-\partial_{x_1}\varphi^3,
\\
a(t)(-2 x_1^3 -2
x_1x_2^2) \varepsilon \partial_{x_2} z^1+a(t)4 x_1 x_2 x_3
\varepsilon  \partial_{x_3}z^1
+a(t)(-2 x_1^3-2 x_1 x_2^2)
\\\varepsilon\partial_{x_1}z^2+a(t)(-4 x_1^2 x_2-4 x_2^3)
\varepsilon  \partial_{x_2}z^2
+a(t)(10 x_1^2 x_3+14 x_2^2 x_3)
\varepsilon  \partial_{x_3}z^2\\
+a(t)(-2 x_1^3 x_3-2 x_1 x_2^2
x_3)\varepsilon  \partial^2_{x_1x_3}z^2+a(t)(-2 x_1^2 x_2 x_3-2 x_2^3 x_3)
\varepsilon  \partial^2_{x_2x_3}z^2\\
+a(t)(4 x_1^2x_3^2+4 x_2^2 x_3^2)
\partial^2_{x_3x_3}z^2
-\partial^2_{x_3t}z^2-\partial^3_{x_1x_1x_3}z^2-\partial^3_{x_2x_2x_3}z^2-\partial^3_{x_3x_3x_3}z^2\\
=\partial_{x_3}\varphi^2-\partial_{x_2}\varphi^3,\\
-\partial_{x_1}z^1-\partial_{x_2}z^2=\varphi^4.
\end{aligned}\right . \label{syst1}
\end{equation}
Let us consider the change of variables $$s:=\varepsilon a(t)$$ and $$e:=\frac{1}{T-t}.$$ ($e$ appears when we differentiate $t\mapsto a(t)$ on $Q_0$). Let $\R[E,S,X]$ be the set of polynomials in the variables $e$, $s$, $x_1$,
$x_2$, $x_3$, with real coefficients. The $30360\times 29900$ entries of ${L_0}$ can alternatively be seen as functions depending on $(t,x_1,x_2,x_3,\varepsilon)$ or as elements of $\R[E,S,X]$ and, from now on, we consider ${L_0}$ as an element
of $\mathcal M_{30360\times 29900}(\R[E,S,X])$.
As we will see after, it turns out that many of the entries of ${L_0}$ are the $0$ polynomial.

 For a positive integer $k$, let us denote by $\mathfrak{S}_k$ the set
 of permutations of $\{1,\ldots, k\}$. To each $\sigma \in \mathfrak{S}_k$, we associate
 the matrix $S_\sigma\in \mathcal M_{k,k}(\R)$ defined by
\begin{gather}
\label{defPsigma}
\begin{aligned}
S_{\sigma(i)i}&=1, \, \forall i \in \{1,\ldots,k\},
\\S_{ji}&=0, \, \forall i \in \{1,\ldots,k\}, \, \forall j \in \{1,\ldots,k\}\setminus\{\sigma(i)\}.
\end{aligned}
\end{gather}

For two positive integers $k$ and $l$, let us denote by
$0_{k\times l}$ the null matrix of $\mathcal M_{k\times l}(\R)$ (which is included in $\mathcal M_{k\times l}(\R[E,S,X])$).
The following lemma is a
key step for the proof of Proposition~\ref{prop-solvable}.

\begin{Lemma}
\label{lemm-good-reordering}
There exist $$\xi^0:=(e^0,s^0,x^0)\in\mathbb R^5,$$ $$\sigma \in \mathfrak{S}_{29900},$$
$${\tilde \sigma} \in \mathfrak{S}_{30360},$$
$$P\in \mathcal M_{7321\times 7321}(\R[E,S,X]),$$ $$Q\in \mathcal M_{23039\times 7321}(\R[E,S,X])$$
and $$R\in \mathcal M_{23039\times 22579}(\R[E,S,X])$$ such that
\begin{gather}
\sigma(i)=i, \,\forall i \in\{1,2,3,4,5,6\},
\label{sigmainvariants}
\\
S_{{\tilde \sigma}}{L_0}S_{\sigma}=
\begin{pmatrix}
P&0_{7321,22579}
\\
Q&R
\end{pmatrix},
\label{M=PQR}
\\
\text{the rank of $P(\xi^0)$ is $7321$.}
\label{Pmaximalrank}
\end{gather}
\end{Lemma}
Let us assume for the moment that this lemma holds and end the proof
of Proposition~\ref{prop-solvable}.  A consequence of Lemma~\ref{lemm-good-reordering} is the following:
\begin{Lemma}
\label{lemmalowerP}
There exists a nonempty open subset $\omega_0$ of $\mathcal C_2$, $T^*>0$ and $\varepsilon^*>0$, such that  \begin{gather}
\label{detnonnulsuromega01}\text {det }P(\frac{1}{T-t},\varepsilon a(t),x)\not = 0,
\, \forall T\in (0,T^*], \, \forall t \in [T/2,T),\, \forall \varepsilon \in (0,\varepsilon_0],
\, \forall x\in \omega_0.
\end{gather}
\end{Lemma}

\textbf{Proof of Lemma~\ref{lemmalowerP}.}

Let us first point out
that $\text {det }P\in \R[E,S,X]$ and, by
\eqref{Pmaximalrank}, this polynomial is not the
$0$ polynomial. Hence there exist a nonnegative integer $m$ and
a polynomial $\tilde P\in \R[E,S,X]$ such that
\begin{gather}
\label{eqtildeP}
\text{det }P(E,S,X)= S^m\tilde P(E,S,X),
\\
\label{tildePnot0}
\tilde P(E,0,X)\in \R[E,X] \text{ is not the $0$ polynomial}.
\end{gather}
By \eqref{tildePnot0}, there exist $\delta >0$, $C'>0$
and a nonempty open subset $\omega_0$ of $\mathcal C_2$ such that
\begin{gather}
\label{ineqtildeP}
|\tilde P (e,0,x)|\geqslant 2 \delta, \, \forall e\in [C',+\infty),\,
\forall x \in \omega_0.
\end{gather}
By the mean value theorem, there exist a positive integer $l$ and a positive real number
$C^*$ such that
\begin{gather}
\label{estdiff}
|\tilde P (e,s,x)-\tilde P(e,0,x)|\leqslant C^*|s|\left(\left|e\right|^l+\left|s\right|^l+1\right),
 \, \forall e\in\R,\, \forall s \in \R, \, \forall x \in \omega_0.
\end{gather}
By \eqref{propa}, there exists $\varepsilon^*$ such that
\begin{gather}
\label{epetit}
\varepsilon^*|a(t)|\left((T-t)^{-l}+\varepsilon^{*l}|a(t)|^l+1\right)
\leqslant \frac{\delta}{C^*}, \, \forall T\in (0,2/C'], \, \forall t\in [T/2,T).
\end{gather}
 From \eqref{ineqtildeP}, \eqref{estdiff} and \eqref{epetit}, we obtain that
\begin{gather*}
|\tilde P ((T-t)^{-1},\varepsilon a(t),x)|\geqslant \delta,
\,  \forall (T, t,\varepsilon,x)\in (0,2/C']\times [T/2,T)\times (0,\varepsilon^*]
\times \omega_0,
\end{gather*}
which concludes
the proof of Lemma~\ref{lemmalowerP}. \cqfd
Let us now go back to the proof of Proposition~\ref{prop-solvable}.
 For every positive integer $l$,
we denote by $\text{Id}_{l}$ the identity matrix of $\R^{l}$. By \eqref{detnonnulsuromega01},
there exists $U\in C^\infty(Q_0; \mathcal{M}_{7321\times 7321}(\R))$ such that
\begin{gather}
\label{Uinvert}
U(t,x)P(t,x,\varepsilon)=\text{Id}_{7321}, \, \forall x \in \omega_0.
\end{gather}
Let $\tilde U \in C^\infty(Q_0; \mathcal{M}_{7321\times 30360}(\R))$
be defined by
\begin{gather}
\label{deftildeU}
\tilde U (t,x):=
\begin{pmatrix}
U(t,x) &0_{7321,23039}
\end{pmatrix}
,\, \forall x \in \omega_0.
\end{gather}
 From \eqref{M=PQR}, \eqref{Uinvert} and \eqref{deftildeU}, one has
\begin{gather}
\label{USMS}
\tilde U(t,x)S_{{\tilde \sigma}}{L_0}(t,x)=
\begin{pmatrix}
\text{Id}_{7321} &0_{7321,22579}
\end{pmatrix}
S_{\sigma}^{-1},\, \forall x \in \omega_0.
\end{gather}
Let $K\in \mathcal M_{6,7321}(\R)$ be defined by
\begin{gather}
\label{defK}
K:=\begin{pmatrix}
\text{Id}_{6} &0_{6,7315}
\end{pmatrix}
.
\end{gather}
 From \eqref{sigmainvariants}, \eqref{USMS} and \eqref{defK}, one has
\begin{gather*}
\label{KUSM}
K\tilde U(t,x)S_{{\tilde \sigma}}{L_0}(t,x)=\begin{pmatrix}
\text{Id}_{6}& 0_{6,29894}
\end{pmatrix} S_{\sigma}^{-1}=
\begin{pmatrix}
\text{Id}_{6}& 0_{6,29894}
\end{pmatrix}, \, \forall x \in \omega_0,
\end{gather*}
which shows that \eqref{NM} holds with $N(t,x):=K\tilde U(t,x)S_{{\tilde \sigma}}$, and ends the proof of Proposition~\ref{prop-solvable}.

\cqfd

To finish the proof of Proposition~\ref{prop-solvable}, it suffices to prove Lemma~\ref{lemm-good-reordering}.

\textbf{Proof of Lemma~\ref{lemm-good-reordering}.}

 The fact that the dependence of $\overline y$ and its derivatives in the time variable is quite complicated (it is both exponential and fractional) compared to the dependence in the space variable (which is polynomial) is problematic, because it is not very convenient to use for computations on a computer. In the previous proof we have seen $\text{det }P$ as a polynomial in $s=\varepsilon a(t)$, $e=\frac{1}{T-t}$ (which corresponds to terms appearing when we differentiate $t\mapsto a(t)$) and $x$. Assume that we fix $e=0$: This is equivalent to do ``as if'' the derivatives of $a$ were all identically the null function, i.e. to do as if  the function $t\mapsto a(t)$ were replaced by a constant function, which is simpler than our original function $a$. We will then impose $e^0=0$ for our computations.
Let us set $\xi^0:=(e^0,s^0,x^0)$ with $e^0=0$, $s^0=1$ and $x^0=(1.1,1.2,1.3)$.

First of all, let us prove that one can decompose $M$ as in \eqref{M=PQR} at least at point $\xi^0$.
We present in the Appendix~\ref{program} how we computed the matrix
$${L_0^0}:={L_0}(\xi^0)\in \mathcal{M}_{30360\times 29900}(\R)$$
 thanks to a $C^{++}$ program.

  From now on we assume that we have
 matrix $L^0_0$ at our disposal and we are going to explain
 how to exploit it in order to obtain Lemma~\ref{lemm-good-reordering}.

 We begin with reordering the columns so
 that the null columns of ${L_0^0}$ are moved to the last columns. One verifies for example thanks to Matlab that there are exactly
 $140$ such columns.
There exist $\overline{\sigma} \in \mathfrak{S}_{29900}$ and
$N^0\in \mathcal{M}_{30360\times 29760}(\R)$ such that

 \begin{gather}
{L_0^0}S_{\overline{\sigma}}=
\begin{pmatrix}
N^0&0_{30360\times 140}
\\
\end{pmatrix}.
\end{gather}
One problem is that it could happen that some columns of ${L_0^0}$
are equal to $0$ but the corresponding columns of ${L_0}$ are not
identically null. However, we check that it is not
the case (thanks to the evaluation function described in Appendix~\ref{program}).

Let us recall that our goal is to extract a well-chosen submatrix of $L^0_0$ which is of maximal rank. A reasonable hope would have been that the matrix $N^0$ (of size $30360\times 29760$) itself is
of maximal rank $29760$ (we would then have obtained something similar to Lemma~\ref{lemm-good-reordering} by
choosing some squared extracted matrix of maximal rank $P^0$ of $N^0$, which is always possible, the matrix $P^0$ would then have been of size $29760\times 29760$ and the non-selected lines would be permuted to obtain matrices $Q^0$ of size $600\times 29760$ and $R^0$ of size $600\times 140$).
 However it turns out to be false, as we will see later.

Since computing the rank of ${L_0^0}$ on a computer is too long because of its size, we introduce the
notion of structural rank.
\begin{Definition}Let  $A\in \mathcal M_{n,m}(\R)$ and $B\in\mathcal M_{n,m}(\R)$. We say that $A$ and $B$ are structurally equivalent  if
the following property is verified:
$$A_{ij}=0\Leftrightarrow B_{ij}=0.$$ This is  an equivalence relation on $\mathcal M_{n,m}(\R)$, and we call $Cl(A)$ the equivalence class of $A$.
The structural rank of a matrix $A$ (denoted $\text{{\tt sprank}}(A)$ in the following)
is the maximal rank of the elements of $Cl(A)$. Equivalently, if we fill randomly the nonzero
coefficients of $A$, then, with probability $1$, the rank of $A$ is equal to the structural rank.
\end{Definition}

One sees that the structural rank does not depend on the coefficients of the
matrix but only on the distribution of the zeros in the matrix and is never less than the rank. The advantage of the structural rank is that
 it can be computed fast (in a couple of seconds in our case),
especially on sparse matrices. It corresponds to the function {\tt sprank} in Matlab.

Computing the structural rank of $N^0$ thanks to Matlab we find that $$\text{{\tt sprank}}(N^0)=28654<29760,$$
hence there is no hope that the rank of $N^0$ is maximal.

To extract a submatrix of $N^0$ which is of maximal rank,
we can, for example, begin with extracting a submatrix of $P^0$ which is of maximal structural rank, and verify that it is of maximal rank too. The right way to do this is to explore more carefully
how the structural rank is computed. In fact the key point is the existence of a decomposition in block triangular form (which is related to the Dulmage-Mendelsohn decomposition for the bipartite graph associated to any matrix, see \cite{MR0097069} and \cite{MR1095132}) of a matrix.
\begin{Proposition}
Let $A$ be a matrix. Then one can permute
the columns and the lines of $A$ to obtain a matrix of the following form:
\begin{equation}\begin{pmatrix}
A_{11} &A_{12} &A_{13}&A_{14}
\\

0 &0 &A_{23}&A_{24}
\\
0 &0 &0&A_{34}
\\
0 &0 &0&A_{44}
\end{pmatrix},\label{block}\end{equation}
where:
\begin{enumerate}
\item $(A_{11}, A_{12})$ is the underdetermined part of the matrix, it always has more columns than rows.
\item $(A_{33}, A_{34})$ is the overdetermined part of the matrix, it always has more rows than columns.
\item $A_{12}, A_{23}, A_{34}$ are square matrices with nonzero diagonals (in particular these matrices are of maximal structural rank)
\item $A_{23}$ is the well-determined part of the matrix (if the matrix is square and non-singular, it is the entire matrix).
\end{enumerate}
Moreover, one can permute rows and columns so that $A_{23}$ is also block triangular.
The decomposition obtained is called the block triangular form of matrix $A$.
The structural rank of $A$ is given by the sum of the structural ranks of $A_{12}, A_{23}, A_{34}$.
\end{Proposition}
The block triangular form \eqref{block} (called the coarse decomposition)
of the matrix is in fact given by the {\tt dmperm} function in Matlab,
which also gives the permutation that makes the matrix be in the form of \eqref{block} and
a block triangular form for the well-determined part
(which is called the fine decomposition).
One easily understands how to obtain a matrix in the form of
\eqref{M=PQR} thanks to this decomposition: One can (for instance) permute the blocks to obtain
\begin{gather}\label{btf-perm}\begin{pmatrix}
A_{34}&0 &0 &0
\\
A_{44}&0 &0 &0
\\
A_{14}&A_{11} &A_{12} &A_{13}
\\
A_{24}&0 &0 &A_{23}
\\
\end{pmatrix},\end{gather}
from which we easily deduce decomposition \eqref{M=PQR}.

To simplify the computations, we are not going to apply this block triangular decomposition directly to ${L_0^0}$ but to
a well-chosen submatrix $\tilde{L}^{0}_0$. First of all, we go back to ${L_0}$  and select some equations and unknowns:
There exist $\sigma^0 \in \mathfrak{S}_{29900}$,
${\tilde \sigma}^0 \in \mathfrak{S}_{30360}$,
$\tilde{Q}\in \mathcal{M}_{16623\times 14630}(\R)$
and $\tilde{R}\in \mathcal{M}_{16623\times  15270}(\R)$
such that (see \eqref{btf-perm})
\begin{gather}\label{presel} S_{{\tilde \sigma}^0}L_{0}S_{\sigma^0}=
\begin{pmatrix}
{\tilde L_{0}}&0_{13737\times 15270}
\\
\tilde{Q}&\tilde{R}
\end{pmatrix}
,\end{gather}
where $\tilde{L}_{0}$ corresponds to the equations we obtain by differentiating the
two first equations of \eqref{eqzphi} 15 times and the last
equation $17$ times, so that $\tilde{L}_{0}$ is of size
$(G(15),H(15))=(13737,14630)$ (here there are more unknowns than
equations but we will see that this will not be a problem).

Let us call $$\tilde{L}^{0}_0:=\tilde{L}_0(\xi^0),$$  
$$\tilde{Q}^{0}:=\tilde{Q}(\xi^0),$$
$$\tilde{R}^{0}:=\tilde{R}(\xi^0).$$
One has
$$S_{{\tilde \sigma}^0}{L_0^0}S_{\sigma^0}=
\begin{pmatrix}
{\tilde {L_0^0}}&0_{13737\times 15270}
\\
\tilde{Q^0}&\tilde{R^0}
\end{pmatrix}
.$$

Thanks to Matlab, we find the Dulmage-Mendelsohn decomposition of $\tilde{L}^{0}_0$ and observe that there exists some permutations matrices
$\sigma^1$ and ${\tilde\sigma}^1$ such that (see \eqref{btf-perm})
$$S_{{\tilde \sigma}^1}\tilde{L}^{0}_0S_{\sigma^1}=
\begin{pmatrix}
\overline{L}^{0}_0&0_{9050\times 5578}
\\
\overline{Q^0}&\overline{R^0}
\end{pmatrix},$$
with $\overline{Q}^0\in \mathcal{M}_{4687\times 9050}(\R)$,
$\overline{R}^0\in \mathcal{M}_{4687\times  5578}(\R)$, and $\overline{L}^{0}_0\in \mathcal{M}_{9050\times  9050}(\R)$
is of maximal structural rank and square 
(it corresponds to the block $A_{34}$ in the block triangular
decomposition). Applying the Dulmage-Mendelsohn algorithm now on
$\overline{L}^{0}_0$, we can write $\overline{L}^{0}_0$ in an (upper) block triangular
form with $352$ diagonal blocks, the first
$351$ of them being of ``small'' size and the latter one being of size $7321$.
Let us call ${\overline{L}^{0}_0}_{(i,j)}$ (with $(i,j)\in [|1;352|]^2$) the blocks of $\overline{L}^{0}_0$.

Using this decomposition, one can see (using Matlab) that $\overline{L}^{0}_0$
is not of maximal rank. However, by computing the rank of the block
${\overline{L}^{0}_0}_{352,352}$ thanks to Matlab, one sees that it is of maximal rank $7321$.
Moreover, we verify that
\begin{equation}
\left \{\begin{aligned}\label{good-place} \text{the columns corresponding to the unknowns }
\partial_{x_1}z^1,\partial_{x_2}z^1,\partial_{x_3}z^1,\partial_{x_1}z^2,\\
\partial_{x_2}z^2,\partial_{x_3}z^2 \text{ appear in this block,}
\end{aligned}
\right .
\end{equation}
 by looking carefully on Matlab
where the columns corresponding to these unknowns have been moved under the action of the permutation matrices $S_{\sigma^0}$ and $S_{\sigma^1}$. More precisely, $\partial_{x_i} z^1$ corresponds to the $i$-th column of ${\overline{L}^{0}_0}_{(352,352)}$ and $\partial_{x_i} z^2$ to the $(3632+i)$-th column of $\overline{L^{0}}_{352,352}$.

 Let us call
$$P^0:={\overline{L}^{0}_0}_{352,352}.$$
There exist $\sigma \in \mathfrak{S}_{29900}$,
${\tilde \sigma} \in \mathfrak{S}_{30360}$,
${Q^0}\in \mathcal M_{23039\times 7321}(\R[E,S,X])$
and
 ${R^0}\in \mathcal M_{23039\times  22579}(\R[E,S,X])$
such that
\begin{gather}\label{okdec}
S_{{\tilde \sigma}}\tilde{L}^{0}_0S_{\sigma}=
\begin{pmatrix}
P^0&0_{7321\times 22579}
\\
Q^0&R^0
\end{pmatrix},
\end{gather}
the rank of the first block $P^0$ being maximal.
The distribution of the nonzero coefficients of $P^0$ is given in Figure $2$.
\begin{figure}[!ht]
\begin{center}
\includegraphics[scale=0.45]{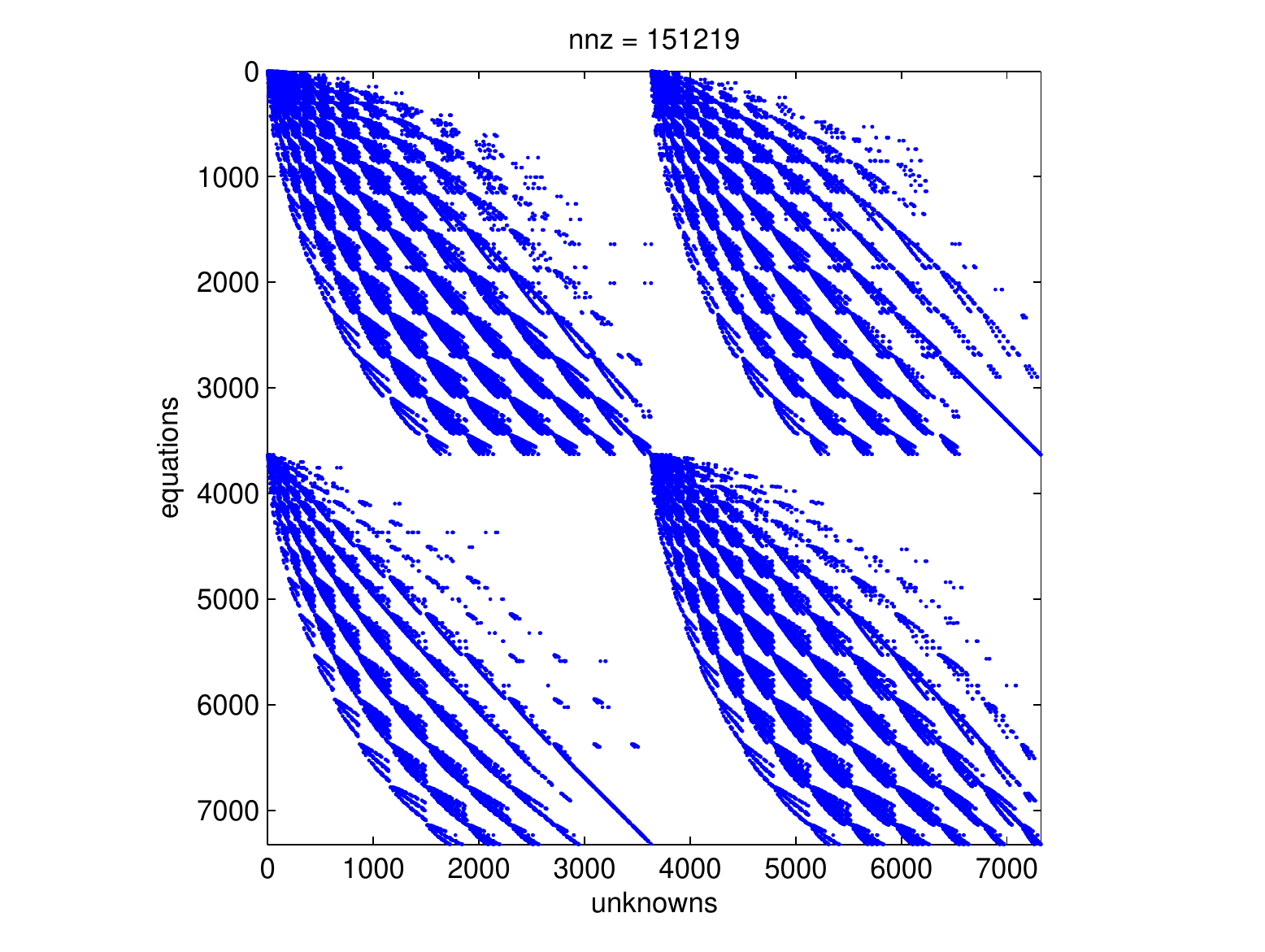}
\caption{Distribution of the nonzero coefficients of $P^0$.}
\end{center}
\end{figure}

Now, we come back to the matrix ${L_0}$ and we consider the following matrix:
$$\tilde{M}^0:=S_{{\tilde \sigma}}\tilde{L}^{0}S_{\sigma},$$
where $S_{{\tilde \sigma}}$ and $S_{{\sigma}}$ are introduced in \eqref{okdec}.

Let us call $\Theta$ the set of all coefficients of ${L_0}$ (considered as a polynomial in the variables $E,S,X$) that are not identically zero, and $\Theta^0$ the set of all coefficients of ${L^0_0}$ that are not equal to $0$. Clearly $\Theta^0\subset \Theta$ (in fact thanks to Matlab one can see that $\Theta^0$ is much smaller than $\Theta$), moreover $\Theta\setminus \Theta^0$ corresponds to the nonzero coefficients of the matrix that become identically null when we change $a(t)$ into the function identically equal to $1$ and apply it at point $\xi^0$. What could happen is that $\tilde{M}^0$ is not of block triangular form as in \eqref{okdec} (the null block of the matrix $\tilde{M}^0$ may contain some elements of $\Theta\setminus\Theta^0$). However, since the only important thing is the location of the elements of $\Theta\setminus\Theta^0$ and not their value, one can verify easily on Matlab that the coefficients of $\Theta\setminus\Theta^0$ do not influence the block form \eqref{okdec} (by looking where the elements of $\Theta\setminus\Theta^0$ are moved under the action of $S_{{\tilde \sigma}}$ and $S_{\sigma}$), i.e. the permutations $S_{{\tilde \sigma}}$ and $S_{\sigma}$ also give a decomposition as in \eqref{okdec} for $\tilde{L}^0$: There exists $$\tilde P\in \mathcal M_{7321\times 7321}(\R[E,S,X]),$$ $$\tilde Q\in \mathcal M_{23039\times 7321}(\R[E,S,X])$$
and $$\tilde R\in \mathcal M_{23039\times 22579}(\R[E,S,X])$$ such that

\begin{gather}\label{okdecabs}
S_{{\tilde \sigma}}\tilde{L}^{0}S_{\sigma}=
\begin{pmatrix}
\tilde P&0_{7321\times 22579}
\\
\tilde Q&\tilde R
\end{pmatrix},
\end{gather}
with the relations $P(\xi^0)=P^0,Q(\xi^0)=Q^0,R(\xi^0)=R^0$.

Property \eqref{M=PQR} follows then directly from \eqref{presel} and \eqref{okdecabs}, \eqref{Pmaximalrank} is a direct consequence of the above construction, and \eqref{sigmainvariants} can be easily deduced by permuting some lines and columns of $S_{{\tilde \sigma}}$ and $S_{\sigma}$ (thanks to Property \eqref{good-place}). Finally, Lemma~\ref{lemm-good-reordering} holds.
\cqfd

Consequently Proposition~\ref{prop-solvable} holds. Moreover,
one observes that
the linear partial differential operator $\mathcal{M}_0$ that
we have just created so that \eqref{eqL0M0} holds is exactly $\mathcal P^*$, (where $\mathcal P$ is the differential version of the matrix $P$ seen as a partial differential operator acting on  $(\partial_{x_1}z^1,\partial_{x_2}z^1,\partial_{x_3}z^1,\partial_{x_1}z^2,\partial_{x_2}z^2,\partial_{x_3}z^2)$) and is of order $17$ (because we have differentiated the equations of \eqref{eqaresoudre} $16$ times, and \eqref{eqzphi} was obtained by differentiating System \eqref{eqaresoudre} $1$ time). Hence, the corresponding operator $\mathcal M$ in equality \eqref{LMB} is also of order $17$.

This concludes the proof of Proposition~\ref{prop-solvable}.
\cqfd

\subsection{Controllability of the linear control system \texorpdfstring{\eqref{linearizedbary-Bf-Cauchy}}{}}
\label{seccontrollabilitywithB}
In this subsection, we prove some technical lemmas that imply the null controllability
of \eqref{linearizedbary-Bf-Cauchy} with controls which are derivatives of smooth enough functions having a small support. This is needed to ensure that the controls are in the image of $\mathcal B$ (this is exactly Assumption \hyperlink{assu1}{$\mathcal{A}_1$}) and to take into account Remark~\ref{not-so-regular}.

The first lemma is the following one (remind that $\overline y$ is one of the trajectories constructed in Section~\ref{secconstrbar}). It consists in a Carleman estimate with curl observation. We call $D_{\overline y}$ the operator
$$D_{\overline y}z:=(\overline y.\nabla)z-
\begin{pmatrix} \partial_{x_1}\overline y^1z^1+\partial_{x_1}\overline y^2z^2+\partial_{x_1}\overline y^3z^3\\
\partial_{x_2}\overline y^1z^1+\partial_{x_2}\overline y^2z^2+\partial_{x_2}\overline y^3z^3\\
\partial_{x_3}\overline y^1z^1+\partial_{x_3}\overline y^2z^2+\partial_{x_3}\overline y^3z^3
\end{pmatrix}.$$
$D_{\overline y}$ is exactly the opposite of the adjoint operator of $y\mapsto (\overline y.\nabla y)+(y.\nabla)\overline y$ (because $\overline y$ is divergence-free).
\begin{Lemma} \label{Carl1}
Let $\theta : \Omega \rightarrow [0,+\infty)$ be a lower semi-continuous function which is not identically $0$ and let $r\in (0,1)$. There exists $C_1>0$ such that, for every $K_1>C_1$, every $\nu\geqslant K_1(1-r)/r$, there exists $\varepsilon^0$ such that for every $\varepsilon\in (0,\varepsilon^0)$, there exists $C>0$ such that, for every $g\in L^2((0,T)\times\Omega)^3$ and for every  solution $z\in L^2((0,T),V)\cap L^\infty((0,T),H)$ of the adjoint of the linearized Navier-Stokes system
\begin{equation}\left\{\begin{aligned}
-z_t-\Delta z-D_{\overline y}z+\nabla \pi&=g&\mbox{ in } Q,
\\\nabla\cdot z&=0&\mbox{ in } Q,
\\z&= 0&\mbox{ on }\Sigma,
\\z(T)&=z^T&\in V,
 \label{Sbis-adj}
\end{aligned} \right.\end{equation}
 one has
\begin{multline}||e^{\frac{-K_1}{2r(T-t)^5}}z||^2_{L^2((T/2,T),H^1(\Omega)^3)}+||z(T/2,\cdot )||_{L^2(\Omega)^3}^2\\\leqslant C\left(\int_{(T/2,T)\times\Omega}\theta e^{-\frac{K_1}{(T-t)^5}}|\nabla\wedge z|^2+\int_{(T/2,T)\times\Omega}e^{-\frac{K_1}{(T-t)^5}}|g|^2\right ) .\label{Carleman-good}
\end{multline}
\end{Lemma}

\textbf{Proof of Lemma~\ref{Carl1}.}

In this proof, our system, which is initially defined on $(0,T)$, will only be considered on the interval of time $(T/2,T)$. In fact, in the following, (see, in particular, the proof of Proposition~\ref{cor-regular-controls}), we will not act on the system on the interval $(0,T/2)$, hence we only need a Carleman estimate on $(T/2,T)$. For our proof, we need to use the particular form of our $\overline y$ in time, in particular that (see \eqref{propa})
\begin{gather}
\label{estimatebary}
|\overline y(t,x)|+|\nabla\overline y(t,x)|\leqslant C\varepsilon e^{\frac{-\nu}{(T-t)^5}}, \, \forall (t,x)\in (T/2,T)\times \Omega.
\end{gather}
Without loss of generality, we may assume that there exists a nonempty open subset $\omega^*$ of $\Omega$ such that $\theta = 1_{\omega^*}$. Let us now give some other notations.
Let $\eta^0\in C^2(\overline\Omega)$ such that $\eta>0$ and $|\nabla \eta^0|>0$ in $\overline\Omega\setminus\omega^*$ and $\eta^0=0$ on $\partial\Omega$.
For the existence of $\eta^0$, see \cite[Lemma 1.1, p.~4]{96FIbook}.
Let us call
$$\alpha(t,x):=\frac{e^{12\lambda ||\eta^0||_\infty}-e^{\lambda (10||\eta^0||_\infty+\eta^0(x))}}{(t-T/2)^5(T-t)^5},\text{    }\xi(t,x):=\frac{e^{\lambda (10||\eta^0||_\infty+\eta^0(x))}}{(t-T/2)^5(T-t)^5}$$ and
$$\alpha^*(t):=\max_{x\in\overline\Omega} \alpha(t,x).$$
%\text{    }\xi^*(t):=\max_{x\in\overline\Omega} \xi(t,x),$$
%$$\widehat\alpha(t):=\min_{x\in\overline\Omega} \alpha(t,x),\text{    }
%\widehat\xi(t):=\min_{x\in\overline\Omega} \xi(t,w).$$
We call $Q_{/2}:=(T/2,T)\times\Omega$.
 Using \eqref{estimatebary} and \cite[Proposition 3.1, p.~6]{gueye2012} on the adjoint system \eqref{Sbis-adj} (where we see the first and zero order terms of this equation as a second member, because Proposition 3.1 of \cite{gueye2012} concerns only the Stokes system), one has, for some $C$ large enough, $\lambda\geqslant C$ and $s\geqslant C$,

\begin{multline*}s^3\lambda^4\int_{Q_{/2}} e^{-2s\alpha-2s\alpha^*}\xi^3|\nabla\wedge z|^2+s\lambda^2\int_{Q_{/2}} e^{-2s\alpha-2s\alpha^*}\xi|\nabla(\nabla\wedge z)|^2\\\leqslant C\left(s^3\lambda^4\int_{(T/2,T)\times\omega^*} e^{-2s\alpha-2s\alpha^*}\xi^3|\nabla\wedge z|^2+\varepsilon^2\int_{Q_{/2}} e^{-2s\alpha^*}e^{-\frac{2\nu}{(T-t)^5}}(|z|^2
\right .\\\left.+|\nabla z|^2)+\int_{Q_{/2}} e^{-2s\alpha^*}|g|^2\right).\end{multline*}

In fact, looking carefully at the proof of Proposition 3.1 of \cite{gueye2012}, one remarks (just by changing the weight $\rho(t):=e^{-s\alpha^*}$ by $\rho(t):=e^{-\mu s\alpha^*}$ where $\mu>1$ is a parameter that can be chosen as large as we wish) that the previous inequality can be improved in the following way, as soon as $s$ is large enough, for every $\mu>1$ (the constant $C$ depends on $\mu$):
\begin{multline}s^3\lambda^4\int_{Q_{/2}} e^{-2s\alpha-2\mu s\alpha^*}\xi^3|\nabla\wedge z|^2+s\lambda^2\int_{Q_{/2}} e^{-2s\alpha-2\mu s\alpha^*}\xi|\nabla(\nabla\wedge z)|^2\\\leqslant C\left(s^3\lambda^4\int_{(T/2,T)\times\omega^*} e^{-2s\alpha-2\mu s\alpha^*}\xi^3|\nabla\wedge z|^2+\varepsilon^2\int_{Q_{/2}} e^{-2\mu s\alpha^*}e^{-\frac{2\nu}{(T-t)^5}}(|z|^2
\right .\\\left.+|\nabla z|^2)+\int_{Q_{/2}} e^{-2\mu s\alpha^*}|g|^2\right).\label{Carlb}\end{multline}

As usual, we now change our weights so that they do not vanish at time $t=T/2$.
Let us call $l:[T/2,T]\rightarrow \mathbb R$ defined by
$l(t)=T^2/16$ on $[T/2,3T/4]$ and
$l(t)=(t-T/2)(T-t)$ on $[3T/4,T]$.
Next, we define
$$\beta(t,x):=\frac{e^{12\lambda ||\eta^0||_\infty}-e^{\lambda (10||\eta^0||_\infty+\eta^0(x))}}{l^5(t)},\text{    }
\gamma(t,x):=\frac{e^{\lambda (10||\eta^0||_\infty+\eta^0(x))}}{l^5(t)},$$
$$\beta^*(t):=\max_{x\in\overline\Omega} \beta(t,x),\text{    }\gamma^*(t):=\max_{x\in\overline\Omega} \gamma(t,x).$$
Clearly, the functions $\alpha$ and $\beta$ coincide on $[3T/4,T]$, as well as the functions $\xi$ and $\gamma$. Using classical energy arguments, we deduce the existence of $C$ (depending now on $s,\lambda$, which are assumed to be large enough and fixed from now on, and $\mu $) such that
\begin{multline}\int_{Q_{/2}} e^{-2(1+\mu )s\beta^*}|\nabla\wedge z|^2+\int_{Q_{/2}} e^{-2(1+\mu )s\beta^*}|\nabla(\nabla\wedge z)|^2\\\leqslant C\left(\int_{(T/2,T)\times\omega^*} e^{-2\mu s\beta^*}\gamma^{*3}|\nabla\wedge z|^2+\varepsilon^2\int_{Q_{/2}} e^{-2\mu s\beta^*}e^{-\frac{2\nu}{(T-t)^5}}(|z|^2+|\nabla z|^2)\right .\\\left .+\int_{Q_{/2}} e^{-2\mu s\beta^*}|g|^2\right)\label{carl-int}.\end{multline}

One remarks that, since $\nabla\cdot z=0$ in $Q$ and $z=0$ on $(0,T)\times\partial\Omega$, one has

\begin{gather}\label{Poin1}C||e^{-(1+\mu )s\beta^*}\nabla\wedge z||_{L^2(Q_{/2})^3}^2\geqslant ||e^{-(1+\mu )s\beta^*}\nabla z||_{L^2(Q_{/2})^9}^2.\end{gather}
Using Poincar\'{e}'s inequality, we also have
\begin{gather}\label{Poin2}C||e^{-(1+\mu )s\beta^*}\nabla z||_{L^2(Q_{/2})^9}^2\geqslant ||e^{-(1+\mu )s\beta^*}z||_{L^2(Q_{/2})^3}^2.\end{gather}
Putting this into \eqref{carl-int}, one obtains
\begin{multline}\int_{Q_{/2}} e^{-2(1+\mu )s\beta^*}|\nabla z|^2+\int_{Q_{/2}} e^{-2(1+\mu )s\beta^*}|z|^2\\\leqslant C\left(\int_{(T/2,T)\times\omega^*} e^{-2\mu s\beta^*}\gamma^{*3}|\nabla\wedge z|^2+\varepsilon\int_{Q_{/2}} e^{-2\mu s\beta^*}e^{-\frac{2\nu}{(T-t)^5}}(|z|^2+|\nabla z|^2)\right .\\\left .+\int_{Q_{/2}} e^{-2\mu s\beta^*}|g|^2\right).
\label{ineqabeta}
\end{multline}
Let us define, for $\mu >1$,
\begin{gather}
\label{defK0}
K_0:=2^{6}(1+\mu +\sqrt{\mu })s\frac{e^{12\lambda ||\eta^0||_\infty}-e^{10\lambda ||\eta^0||_\infty}}{T^5},
\\
\label{defK1}
K_1:=2^{6}(\mu -\sqrt{\mu })s\frac{e^{12\lambda ||\eta^0||_\infty}-e^{10\lambda ||\eta^0||_\infty}}{T^5}.
\end{gather}
From equality \eqref{defK1}, one deduces the existence of $C^*>0$ (depending on $\mu >1$, $\lambda>>1$ and $s>>1$) such that
\begin{gather}
\label{comp-poids-ok-K1}
e^{-2\mu s\beta^*(t)}\gamma^{*3}(t)\leqslant C^*e^{\frac{-K_1}{(T-t)^5}}, \, \forall t\in (T/2,T).
\end{gather}
Moreover, from equality \eqref{defK0}, there exists $\widehat C>0$ (depending on $\mu >1$, $\lambda>>1$ and $s>>1$) such that
\begin{gather}
\label{comp-poids-ok-K0}
e^{-\frac{K_0}{(T-t)^5}}\leqslant \widehat C e^{-2(1+\mu )s\beta^*}, \, \forall t\in (T/2,T).
\end{gather}
 Fixing $s,\lambda$ and making  $\mu \rightarrow + \infty$, one easily sees that $K_0/K_1=\frac{1+\mu +\sqrt{\mu }}{\mu -\sqrt{\mu }}\rightarrow 1^+$ so that for every $r\in(0,1)$, we have for $\mu $ large enough,  $K_0<K_1/r$.  For $\varepsilon>0$ small enough and for $\nu$ large enough ($\nu\geqslant K_0-K_1$), one can absorb the undesired terms $\varepsilon^2\int_{Q_{/2}} e^{-2\mu s\beta^*}e^{-\frac{2\nu}{(T-t)^5}}(|z|^2+|\nabla z|^2)$ from the right-hand side of \eqref{ineqabeta}. Then using some classical energy estimates together with \eqref{comp-poids-ok-K1} and \eqref{comp-poids-ok-K0}, one obtains \eqref{Carleman-good}.
\cqfd

 From now on, we set
\begin{gather*}
\rho_r(t):=e^{\frac{-K_1}{r(T-t)^5}}, \, \rho_1(t):=e^{\frac{-K_1}{(T-t)^5}}.
\end{gather*}
Let us now derive from this Carleman inequality a result of null-controllability with controls which are derivatives of smooth functions.
Let $\widehat 1_{\omega_0}: \R^3\rightarrow [0,1]$ be a function of class $C^\infty$ which is not identically equal to $0$ and
having a support included in $\omega_0$, where $\omega_0$ was introduced in Lemma~\ref{lemmalowerP}. We apply Lemma~\ref{Carl1} with
$\theta = \widehat 1_{\omega_0}$.  One has the following proposition.

\begin{Proposition}
\label{cor-regular-controls}
With the notations of Lemma~\ref{Carl1}, let $f\in L^2(Q)^3$ be such that ${\rho_r}^{-1/2}f\in L^2(Q)^3$ and let
us consider the following linearized Navier-Stokes control system
\begin{equation}\left\{\begin{aligned}
y_t-\Delta y+(\overline y\cdot \nabla)y+(y\cdot \nabla)\overline y+\nabla p&=f+\nabla\wedge((\nabla\wedge v)\widehat 1_{\omega_0})&\mbox{ in } Q,
\\\nabla\cdot y&=0&\mbox{ in } Q,
\\y&= 0&\mbox{ on }\Sigma, \label{Sbis-curl}
\end{aligned} \right.\end{equation}
where the control is $v$. Then, for every $y_0\in V$, there exists a solution  $(y,p,v)$ of $\eqref{Sbis-curl}$  such that $y(0,\cdot )=y^0$ and for every  $\tilde{K}_1$ verifying $0<\tilde{K}_1<K_1$,
\begin{gather}
\label{bonne-reg}
e^{\frac{\tilde{K}_1(2-1/r)}{2(T-t)^5}}(\nabla\wedge v){\widehat 1_{\omega_0}}\in L^2((0,T),H^{53}(\Omega)^3)\cap H^{27}((0,T),H^{-1}(\Omega)^3),
\\
e^{\frac{\tilde{K}_1}{2(T-t)^5}}y\in L^2((0,T),H^{2}(\Omega)^3)\cap L^{\infty}((0,T),H^1(\Omega)^3).
\end{gather}

\end{Proposition}
\begin{Remark}\label{Launching-flowers}What is important in the previous proposition is the fact the controls are very regular (which is quite new and interesting in itself) and that the controls are derivatives (in fact curls) of functions, as in \cite{MR2371118}. In the following, it is enough to obtain a regularity $L^2((0,T),H^{53}(\Omega)^3)\cap H^{27}((0,T),H^{-1}(\Omega)^3)$ for $e^{\tilde{K}_1(2-1/r)/(2(T-t)^5)}(\nabla\wedge v){\widehat 1_{\omega_0}}$ but the following proof can be easily adapted to deduce controls $v$ with $$e^{\tilde{K}_1(2-1/r)/(2(T-t)^5)}(\nabla\wedge v){\widehat 1_{\omega_0}}\in L^2((0,T),H^{2m+1}(\Omega)^3)\cap H^{m+1}((0,T),H^{-1}(\Omega)^3)$$ for every given
$m$ as large as one wants.
\end{Remark}

\textbf{Proof of Proposition~\ref{cor-regular-controls}.}

 In the following, we only control on the interval of time $(T/2,T)$, i.e. we set $v=0$ on $(0,T/2)$ and let the corresponding solution $(y,p)$ of \eqref{Sbis-curl} on $(0,T/2)$ evolve naturally until time $T/2$. Let $y^{T/2}=y(T/2, \cdot)$.

Let  $P:L^2(0,L)^3\rightarrow L^2(0,L)^3$ be the Leray projector $P\varphi:=\varphi -\nabla p$, where $\Delta p=\text{ div }
\varphi $ in $\Omega$ and $\partial p/\partial n=\varphi \cdot n$ on $\partial \Omega$, ($n$ is the unit outward normal vector on $\partial \Omega$). Since $P\Delta \varphi = \Delta P\varphi$ for every $\varphi \in C^\infty_0(\Omega)^3$, $P$ can be extended as a continuous linear map from $H^{-1}(\Omega)^3$ to $H^{-2}(\Omega)^3$. We still denote by $P$ this extension. Let $S: \mathcal{D}'((T/2,T),H^1_0(\Omega)^3)\rightarrow
\mathcal{D}'((T/2,T), H^{-2}(\Omega)^3)$  and $S^*: \mathcal{D}'((T/2,T); H^1_0(\Omega)^3)\rightarrow
\mathcal{D}'((T/2,T), H^{-2}(\Omega)^3)$ be defined by
\begin{gather}
\label{defN}
Sz:=-z_t-P\left(\Delta z+D_{\overline y}z\right),
\\
\label{defN*}
S^*z:=z_t-P\left(\Delta z-(\overline y\cdot\nabla) z - (z\cdot\nabla) \overline y\right).
\end{gather}
($S^*$ corresponds to the linearized time-dependent Navier-Stokes operator and, formally, $S$ is the adjoint of $S^*$).

Since $y^{T/2}$ is regular enough, one can assume from now on without loss of generality that $y^{T/2}=0$ by adding some suitable term in the source term $f$ (that we still call $f$) that still satisfies ${\rho_r}^{-1/2}f\in L^2(Q)^3$, and one can always assume that $P f=f$ by changing the pressure.
We define a closed linear unbounded operator $\mathcal{S}: L^2(Q_{/2})^3\rightarrow L^2(Q_{/2})^3$ by
\begin{gather}
\label{domainecalN}
\mathcal D(\mathcal S):=\{z\in L^2((T/2,T),H^1_0\cap H^2(\Omega)^3)\cap H^1((T/2,T),L^2(\Omega)^3)|z(T,\cdot)=0\},
\\
\mathcal{S}z:=-z_t-P\left(\Delta z+D_{\overline y} z\right).
\end{gather}
We call
$$X_m:=\mathcal D({\mathcal S}^m)$$
and $$X_{-m}:=X_m',$$
where the pivot space is $L^2(Q_{/2})^3$.
For every $(k,l)\in \mathbb Z^2$ such that $k\leqslant l$, one has
$$X_l\subset X_k.$$
Moreover $X_m$ is an Hilbert space for the scalar product
$$<z_1,z_2>_{X_m}:=<\mathcal S ^mz_1,\mathcal S^mz_2>_{L^2(Q_{/2})^3}.$$
The associated norm is denoted $||.||_{X_m}$.
For $m\in\mathbb N$, one can define $\mathcal{S}^*$ as an operator from $X_{-m}$ into $X_{-m-1}$ by setting, for every $z_1\in X_{-m-1}$ and $z_2\in X_{m+1}$,
\begin{gather}
\label{defdualL}<\mathcal{S}^*z_1,z_2>_{X_{-m-1},X_{m+1}}:=<z_1,\mathcal S z_2>_{X_{-m},X_m}.\end{gather}
(One easily checks that this definition is consistent: it gives the same image if $z_1$ is also in $X_{-m'}$ for some $m'\in \mathbb{N})$.
This implies in particular that, for every $z_1\in L^2(Q_{/2})^3$ and for every $z_2\in X_m$, one has, for every $0\leqslant j\leqslant l$,
\begin{gather}\label{dualLkl}<({\mathcal S}^*)^lz_1,z_2>_{X_{-l},X_{l}}=<({\mathcal S}^*)^{l-j} z_1,(\mathcal S)^jz_2>_{X_{j-l},X_{l-j}}.
\end{gather}
Let $\mathcal{H}_0$ be the set of
 $z\in H^1((T/2,T), L^2(\Omega)^3)\cap L^2((T/2,T),H^2(\Omega)\cap V)$  such that
\begin{gather}
\label{condition1}
\sqrt{\rho_1} S z\in X_{26},
\\
\label{condition2}
\sqrt{\widehat 1_{\omega_0}\rho_1}(\nabla\wedge z) \in L^2(Q/2)^3.
\end{gather}
Let $a$ be the following bilinear form defined on $\mathcal{H}_0$:
$$
a(z,w):=<\sqrt{\rho_1}S z,\sqrt{\rho_1}S w>_{X_{26}}+\int_{Q_{/2}}\widehat 1_{\omega_0}\rho_1(\nabla\wedge z).(\nabla\wedge w).$$
 From  \eqref{Carleman-good}, we deduce that $a$ is a scalar product on $\mathcal{H}_0$. Let $\mathcal{H}$ be the completion of  $\mathcal{H}_0$ for this scalar product. Note that, still from \eqref{Carleman-good} and also from the definition of
 $\mathcal{H}$,  $\mathcal{H}$ is a subspace of $L^2_{loc}([T/2,T),H^1_0(\Omega)^3)$ and, for every $z\in \mathcal{H}$,
 one has \eqref{condition1},  \eqref{condition2} and
\begin{gather}
\label{estzH}
||{\rho_r}^{1/2}z||_{L^2((T/2,T),H^1(\Omega)^3)}\leqslant C\sqrt{a(z,z)}, \mbox{ }\forall z\in \mathcal{H}.
\end{gather}
 Let us now consider the linear form $l$ defined on $\mathcal{H}$  by
$$l(w):=\int_{Q_{/2}} fw.$$ The linear form $l$ is well-defined and continuous on $\mathcal{H}$ since,
 by the Cauchy-Schwarz inequality together with \eqref{estzH}, one has, for every $w\in \mathcal{H}$,
\begin{equation}
\left\{\begin{aligned}\label{wpLM}\int_{Q_{/2}} |fw|
&\leqslant ||{\rho_r}^{-1/2}f||_{L^2(Q_{/2})^3}||{\rho_r}^{1/2}w||_{L^2(Q_{/2})^3}
\\\mbox{ }&\leqslant C||{\rho_r}^{-1/2}f||_{L^2(Q_{/2})^3}\sqrt{a(w,w)}.\end{aligned} \right . \end{equation}
Applying the Riesz representation theorem, there exists a unique
\begin{gather}
\label{hatzinH}
\widehat z\in \mathcal{H}
\end{gather}
 verifying, for every
$w\in \mathcal{H}$,

\begin{gather}\label{Lax-M}
<\mathcal S^{26}(\sqrt{\rho_1}S\widehat z),\mathcal S^{26}(\sqrt{\rho_1} S w)>_{L^2(Q_{/2})^3}-\int_{Q_{/2}} \widehat u  w=\int_{Q_{/2}} fw,
\end{gather}
with
\begin{gather}
\label{defhatu}
\widehat u:= -\rho_1\nabla\wedge(\widehat 1_{\omega_0}\nabla\wedge \widehat z).
\end{gather}
We then set
\begin{gather}
\label{def-tildey}
\tilde y:= (\mathcal S^*)^{26}\mathcal {S}^{26}(\sqrt{\rho_1} S \widehat z)\in X_{-26}.
\end{gather}
We want to gain regularity on $\tilde y$ (by accepting to have a weaker exponential decay rate for $\tilde y$ when $t$ is close to $T$).
Let $\psi \in C^\infty([T/2,T])$ and $y\in X_{-1}$. One can define $\psi y\in X_{-1}$ by the following way. Since $\mathcal S^*:X_{0}\rightarrow X_{-1}$ is onto, there exists  $h\in X_0$ such that
$\mathcal S^*h=y$. We define $\psi y $ by
\begin{gather}
\psi y= \psi \mathcal S^*h:=\psi 'h-\mathcal S^*( \psi h).
\end{gather}
This definition is compatible with the usual definition of $\psi y$ if $y\in X_0$. We can then define by induction on $m$
$\psi y\in X_{-m}$ for $\psi \in C^\infty([T/2,T])$ and $y\in X_{-m}$ in the same way.
Using \eqref{def-tildey}, this allows us to define
\begin{gather}
\label{defhaty}
\widehat y:= \sqrt{\rho_1} \tilde y\in X_{-26}.
\end{gather}
 From \eqref{Lax-M}, \eqref{defhatu}, \eqref{def-tildey} and \eqref{defhaty}, one gets
\begin{gather}
\label{eqhaty}
 \mathcal{S}^* \widehat y= f +\widehat u \text{ in }X_{-27}.
\end{gather}
Let $\tilde K_1\in(0,K_1)$ and
$\tilde \rho_1:=e^{-\tilde{K}_1/(T-t)^5}$. Using \eqref{defhatu}, \eqref{def-tildey} and \eqref{eqhaty}, one has
\begin{gather}
\label{eqN*y}
\mathcal{S}^* \left(\left(\sqrt {\rho_1}/\sqrt {\tilde \rho_1}\right)\tilde y\right)
=\left(1/\sqrt{\tilde {\rho_1}}\right)'\sqrt{\rho_1} \tilde y +\left(1/\sqrt {\tilde \rho_1}\right) (f+\widehat u) \in X_{-26}.
\end{gather}
We want to deduce from \eqref{eqN*y} that $\tilde y$ is more regular. This can be achieved thanks to the following lemma:
\begin{Lemma}\label{reg-dual} Let $m\in \mathbb{N}$.
If $y\in X_{-m}$ and $\mathcal S^*y\in X_{-m}$, then $y\in X_{-m+1}$.
\end{Lemma}
\textbf{Proof of Lemma~\ref{reg-dual}. }

If $m=0$, Lemma~\ref{reg-dual} follows from usual estimates on usual regularity property of solutions of the linearized Navier-Stokes system.  From now on, we assume that $m\in \mathbb{N}^*$. Let $h\in X_m$.  Since $\mathcal S$ is an operator from $X_{m+1}$ onto $X_{m}$, there exists $\alpha\in X_{m+1}$ such that
$\mathcal{S} \alpha=h$. Thanks to \eqref{defdualL}, one has
\begin{gather}\label{iop}
<y,h>_{X_{-m},X_{m}}=<\mathcal S^* y,\alpha>_{X_{-m-1},X_{m+1}}=<\mathcal S^*y,\alpha>_{X_{-m},X_{m}},\end{gather}
the last equality coming from the fact that $\mathcal{S}^*y\in X_{-m}$. We deduce from \eqref{iop} that there exists some constant $C>0$ such that for every $h\in X_m$,
\begin{gather}\label{goodxm}|<y,h>|_{X_{-m},X_{m}}\leqslant C ||\alpha||_{X_{m}}=C||h||_{X_{m-1}},\end{gather}
which shows that $y\in X_{1-m}$. This concludes the proof of Lemma~\ref{reg-dual}.
\cqfd
 From \eqref{eqN*y} and Lemma~\ref{reg-dual}, one gets that
$$\left(\sqrt{\rho_1}/\sqrt {\tilde \rho_1}\right)\tilde y\in X_{-25},\,\forall \tilde K_1\in (0,K_1) .$$
Using an easy induction argument together with Lemma~\ref{reg-dual} (and the fact that one can choose $\tilde K_1<K_1$ arbitrarily close to $K_1$), we deduce that, for every $\tilde K_1\in (0,K_1)$,
$\left(\sqrt{\rho_1}/\sqrt {\tilde \rho_1}\right)\tilde y\in X_{0}$.

Let us now focus on $\widehat u$. Let us call $v:=\rho_1 \widehat z$. Using \eqref{estzH}, one gets that
\begin{gather}\label{vx0}{\rho_1}^{-1}{\rho_r}^{1/2}v\in L^2(Q_{/2}).\end{gather}
Using \eqref{hatzinH} together with regularity results for $S$ applied on ${\tilde \rho_1}^{-1}{\rho_r}^{1/2}v\in L^2(Q_{/2}) $ and, as above for the proof of \eqref{vx0}, a bootstrap argument (together with the fact that one can choose $\tilde K_1\in (0,K_1)$ arbitrarily close to $K_1$), one obtains that
\begin{gather}
\label{reg-cont-ok}{\tilde\rho_1}^{-1}{\rho_r}^{1/2}v\in X_{27}, \, \forall \tilde K_1\in (0,K_1).
\end{gather}
From \eqref{reg-cont-ok} and \eqref{eqhaty}, we deduce (by looking the equation verified by $(1/\sqrt {\tilde \rho_1}) \widehat y$ and using usual regularity results for linearized Navier-Stokes system) that 
\begin{gather}
\label{reg-sol-ok}
\left(1/\sqrt {\tilde \rho_1}\right) \widehat y\in L^2((T/2,T),H^{2}(\Omega)^3)\cap L^{\infty}((T/2,T),H^1_0(\Omega)^3), \, \forall \tilde K_1\in (0,K_1).
\end{gather}
Proposition~\ref{cor-regular-controls}  follows from \eqref{eqhaty}, \eqref{reg-sol-ok} and \eqref{reg-cont-ok}.
\cqfd

\subsection{Null-controllability of \texorpdfstring{\eqref{linearizedbary}}{}}
\label{recolle}
To finish, one can gather the results of Subsection~\ref{secalgebraicsolvability} and Subsection~\ref{seccontrollabilitywithB} in order to apply Proposition~\ref{prop-alg-controllability} and obtain a controllability result on \eqref{linearizedbary}. However, we cannot work in the $C^\infty$ setting of Proposition~\ref{prop-alg-controllability}, so we need to take into account Remark~\ref{not-so-regular} and to be careful concerning the spaces we are working with.
\label{recolle-morceaux}
\begin{Proposition}
\label{cor-une-seule-composante}
 For every $T>0$ small enough, for every $\alpha\in(0,1)$, there exists $r_0\in (0,1)$ such that for every $r\in(r_0,1)$, there exists $C_1>0$ such that for every $K_1>C_1$, for every $f\in L^2(Q)$ be such that $e^{\frac{K_1}{2r(T-t)^5}}f\in L^2(Q)^3$ and for every $y^0\in V$, if
\begin{gather}
\label{defnu}
\nu = \frac{1-r}{r}K_1,
\end{gather}
there exists a solution $(y,p,v)$ of the following linearized Navier-Stokes control system
\begin{equation}\left\{\begin{aligned}
y_t-\Delta y+(\overline y\cdot \nabla)y+(y\cdot \nabla)\overline y+\nabla p&=f+(0,0,1_{\omega_0}v)&\mbox{ in } Q,
\\\nabla\cdot y&=0&\mbox{ in } Q,\\
y(0,\cdot)&=y^0&\mbox{ in } \Omega,\\
y&= 0&\mbox{ on }\Sigma \label{Sbis-final},
\end{aligned} \right.\end{equation}
such that
\begin{gather}
e^{\frac{\alpha{K}_1}{2(T-t)^5}}y\in  L^2((0,T),H^{2}(\Omega)^3)\cap L^\infty((0,T),H^1(\Omega)^3), \label{decr-y}
\\ p\in L^2((0,T),H^1(\Omega)), \, e^{\frac{\alpha{K}_1}{2(T-t)^5}}v\in  L^2(Q).
\label{decr-w}
\end{gather}
\end{Proposition}

\textbf{Proof of Proposition~\ref{cor-une-seule-composante}.}

We want to apply Proposition~\ref{prop-alg-controllability}. First of all, we deal with
Assumption \hyperlink{assu1}{$\mathcal{A}_1$}. We apply Proposition~\ref{cor-regular-controls}: There exists a solution  $(y^*,p^*,v^*)$ of $\eqref{Sbis-curl}$   such that $y^*(0,\cdot )=y^0$ and, for every  $\tilde{K}_1$ verifying $0<\tilde{K}_1<K_1$,
\begin{gather}
e^{\frac{\tilde{K}_1(2-1/r)}{2(T-t)^5}}{\widehat 1_{\omega_0}}(\nabla\wedge v^*)\in L^2((T/2,T),H^{53}(\Omega)^3)\cap H^{27}((T/2,T),H^{-1}(\Omega)^3)\label{vetreg},
\\
e^{\frac{\tilde{K}_1}{2(T-t)^5}}y^*\in L^2((T/2,T),H^{2}(\Omega)^3)\cap L^{\infty}((T/2,T),H^1(\Omega)^3)\label{yetreg}.
\end{gather}

Using well-known interpolation results (see for example \cite[Section 13.2, p.~96]{1968-Lions-Magenes-I}) and  setting $n:=27$, we obtain that  $$e^{\frac{\tilde{K}_1(2-1/r)}{2(T-t)^5}}{\widehat 1_{\omega_0}}(\nabla\wedge v^*)\in H^{2n/3}(Q_0)\subset H^{17}(Q_0).$$
Let us call $w:=\widehat 1_{\omega_0}(\nabla\wedge v^*)$, which is supported in $Q_0$.
One observes that
$$\nabla\wedge w=\begin{pmatrix}\partial_{x_3}w^{2}-\partial_{x_2}w^{3}\\ \partial_{x_3}w^{1}-\partial_{x_1}w^{3}\\ \partial_{x_2}w^{1}-\partial_{x_1}w^{2}
\end{pmatrix}.$$
Hence in view of equality \eqref{defmathcalB} and setting
$$f^1=0,f^2=-w^{3},f^3=w^{2},f^4=-w^{3},f^5=0,f^6=w^{1},f^7=\partial_{x_2}w^{1}-\partial_{x_1}w^{2},$$
one has $\nabla\wedge w\in Im(\mathcal B)$ and Assumption \hyperlink{assu1}{$\mathcal{A}_1$} holds.

Now, we observe that Assumption \hyperlink{assu2}{$\mathcal{A}_2$} follows from Proposition~\ref{prop-solvable}. Let $(\tilde y,\tilde p, \tilde v)$
be defined by
$$(\tilde y,\tilde p, \tilde v):=-\mathcal M w,$$
where $\mathcal M$ is as in (the proof of) Proposition~\ref{prop-solvable}. It makes sense to apply $\mathcal M$ to $w$ because $\mathcal M$ is a partial differential operator of order $17$ and $w\in H^{17}(Q_0)$.

Using the fact that the operator $\mathcal{M}$ is a partial differential operator of order $17$ and  that the coefficients of $\mathcal M$ explode at time $t=T$ at rate at most $e^{\frac{7321 \nu}{(T-t)^5}}$, as it follows from the construction of $\overline y$ given in Section~\ref{secconstrbar} (see in particular \eqref{propa}) and the construction of $M$ given in the proof of Proposition~\ref{prop-solvable}, one has
\begin{gather}\label{vreg}e^{\frac{K_2}{2(T-t)^5}}\tilde v\in L^2(Q),\\
\label{yreg}e^{\frac{K_2}{2(T-t)^5}}\tilde y\in L^2(Q),\end{gather}
for every $K_2<K_1(2-1/r)-7321 \nu$. In order to be obtain $\tilde y(T,.)=0$, it is enough to have an exponential decay for $y$ at time $T$, i.e. to impose $$K_1(2-1/r)-7321 \nu>0,$$ which, with \eqref{defnu}, is equivalent to
$$r>\frac{7322}{7323},$$
which can be ensured since $r$ can be arbitrarily chosen close to $1$.
Let $\alpha \in (0,1)$. We set
\begin{gather}
\label{defr0}
r_0:=\frac{7322}{7323-\alpha}.
\end{gather}
Then, if $r\in (r_0,1)$, one has
\begin{gather}
\label{intervalnonvide}
\alpha K_1< K_1- 7321K_1\frac{1-r}{r}.
\end{gather}
By \eqref{intervalnonvide}, there exists $K_2$ such that
\begin{gather}
\label{encadrementK2}
\alpha K_1<K_2< K_1- 7321K_1\frac{1-r}{r}= K_1- 7321\nu.
\end{gather}
Finally, one can apply (the proof of) Proposition~\ref{prop-alg-controllability} and we set $$(y,p,v):= (y^*+\tilde y,p^*+\tilde p, v^*+\tilde v).$$
Thanks to \eqref{vetreg}, \eqref{yetreg}, \eqref{vreg} and \eqref{yreg}, one has
\begin{gather*}
e^{\frac{K_2}{2(T-t)^5}}v\in L^2(Q),
\\
e^{\frac{K_2}{2(T-t)^5}}y\in L^2(Q).
\end{gather*}
Then, using usual regularity results for the linearized Navier-Stokes operators on $e^{\frac{K_2}{2(T-t)^5}}y$ (now considered on the entire time interval $(0,T)$), we obtain
$$e^{\frac{\alpha{K}_1}{2(T-t)^5}}y\in  L^2((0,T),H^{2}(\Omega)^3)\cap L^\infty((0,T),L^2(\Omega)^3),$$
as soon as $y^0\in V$.
The proof of Proposition~\ref{cor-une-seule-composante} is completed.
\cqfd
\section{Proof of Theorem~\ref{th-STLC}}
\label{IL}To conclude, we are going to apply an inverse mapping theorem to go back to the nonlinear system, which is the following (see \cite[Chapter 2, Section 2.3]{MR924574}):
\begin{Proposition}
\label{prop-inverse-mapping}
Let $E$ and $F$ be two Banach spaces. Let $e_0\in E$ and $\mathcal F:E\rightarrow F$ which is of class $C^1$ in a neighborhood of $e_0$. Assume that the operator $d\mathcal F(e_0)\in\mathcal L_c(E,F)$ is onto. Then there exist $\eta>0$ and $C>0$ such that for every $g\in F$ verifying  $||g-\mathcal F(e_0)||<\eta$, there exists $e\in E$ such that
\begin{enumerate}
\item $\mathcal F(e)=g$,
\item $||e-e_0||_E\leqslant C||g-\mathcal F(e_0)||_F$.
\end{enumerate}
\end{Proposition}
 We are going to use the same techniques as in \cite{MR2558423}.
Let $\alpha\in (0,1)$, and let us consider some $r\in (r_0,1)$ where $r_0$ verifies \eqref{defr0}.
We apply Proposition~\ref{prop-inverse-mapping} with $E$ and $F$ defined in the following way.
Let $E$ be the space of the functions
$$(y,p,v)\in L^2(Q)^3\times L^{2}(Q)\times L^{2}(Q)$$ such that
\begin{enumerate}
\item $e^{\frac{\alpha{K}_1}{2(T-t)^5}}y\in L^\infty((0,T),V)^3\cap L^{2}((0,T),H^2(\Omega)^3\cap V)$,
\item $\nabla p\in L^2(Q)$,
\item $e^{\frac{\alpha{K}_1}{2(T-t)^5}}v\in L^{2}(Q)^3$ and the support of $v$ is included in $Q_0$,
\item $e^{\frac{K_1}{2r(T-t)^5}}(y_t-\Delta y
+(\overline y\cdot \nabla)y+(y\cdot \nabla)\overline y +\nabla  p-(0,0,v))\in L^2(Q)^3$,
\item $y(0,\cdot )\in V$,
\end{enumerate}
 equipped with the following norm which makes it a Banach space:
\begin{multline*}
||(y,p,v)||_E:= ||e^{\frac{\alpha{K}_1}{2(T-t)^5}}y||_{L^\infty((0,T),H^{1}_0(\Omega)^3)\cap L^{2}((0,T),H^2(\Omega)^3)}\\+||p||_{L^2((0,T),H^1(\Omega))}+ ||e^{\frac{\alpha{K}_1}{2(T-t)^5}}v||_{L^2(Q)^3}\\+||e^{\frac{K_1}{2r(T-t)^5}}(y_t-\Delta+(\overline y\cdot \nabla)y+(y\cdot \nabla)\overline y+\nabla p-(0,0,v))||_{L^2(Q)^3}\\+||y(0,\cdot )||_{H_0^1(\Omega)^3}.
\end{multline*}
Let $F$ be the space of the functions $(h,y^0)\in L^2(Q)^3\times V$ such that
$$
e^{\frac{\alpha{K}_1}{2(T-t)^5}}h\in L^2(Q)^3,
$$
equipped with the following scalar product which makes it a Hilbert space:
$$((h,y^0)|(k,z^0))=(e^{\frac{\alpha{K}_1}{2(T-t)^5}}h|e^{\frac{\alpha{K}_1}{2(T-t)^5}}k)_{L^{2}(Q)^3}+(y^0|z^0)_{H_0^1(Q)^3}.$$
We define $$\mathcal F(y,p,v)=(y_t-\Delta y+(y\cdot \nabla)y+(\overline y\cdot \nabla)y+(y\cdot \nabla)\overline y+\nabla p-(0,0,v),y(0,\cdot )).$$
To apply the previous inverse mapping theorem, we first show the following lemma.
\begin{Lemma}
\label{lem-A-regular}
The map $\mathcal F$ has its image included in $F$ and is of class $C^1$ on $E$.
\end{Lemma}

\begin{proof}
We see that $\mathcal F= {\mathcal F}_1+ {\mathcal F}_2$ with
$${\mathcal F}_1(y,p,v):=(y_t-\Delta y+(\overline y\cdot \nabla)y+(y\cdot \nabla)\overline y+\nabla p-(0,0,v),y(0,\cdot )).$$
and
$${\mathcal F}_2(y,p,v):=((y\cdot \nabla)y,0).$$

Thanks to the construction of $E$ and $F$ we have
${\mathcal F}_1: E\rightarrow F$ and ${\mathcal F}_1$ is continuous,
so, since ${\mathcal F}_1$ is linear, ${\mathcal F}_1$ is of class $C^1$. The map ${\mathcal F}_2$
is a quadratic form, hence to prove that it maps $E$ into $F$ and is of class $C^1$,
it is sufficient to prove that it is continuous, i.e.  to prove that
\begin{gather}
\label{tobeprovedforC}
||e^{\frac{K_1}{2r(T-t)^5}}(y\cdot \nabla)y||_{L^2(Q)^3}\leqslant C ||(y,p,v)||^2_E.
\end{gather}
We choose $r$ and $\alpha$ so that
\begin{gather}
\frac{K_1}{r}<2\alpha{K}_1 \label{poidsproches}.
\end{gather}
(One can take for example $\alpha =3/4$ and $r\in (0,1)$ close enough to $1$)
Let us call
\begin{gather}
\label{deftildeyexp}
\tilde y(t,x):=e^{\frac{K_1}{4r(T-t)^5}}y.
\end{gather}
This definition of $\tilde{y}$ and inequality \eqref{poidsproches} imply that
$$||\tilde{y}||_{L^\infty((0,T),H^1(\Omega)^3)}\leqslant C ||(y,p,v)||_E,$$
which gives that
\begin{gather}
\label{ineqnabla}
||\nabla \tilde{y}||_{L^\infty((0,T),L^2(\Omega)^9)}\leqslant C ||(y,p,v)||_E.
\end{gather}
We also have
\begin{gather}
\label{ineqtildeyL2H2}
||\tilde{y}||_{L^2((0,T),H^2(\Omega)^3)}\leqslant C ||(y,p,v)||_E.
\end{gather}
A classical Sobolev embedding in dimension $3$ together with \eqref{ineqtildeyL2H2} imply that
\begin{gather}
\label{ineqtildeyL2Linfty}
||\tilde{y}||_{L^{2}((0,T),L^\infty(\Omega)^3)}\leqslant C ||(y,p,v)||_E.
\end{gather}
Direct computations imply that
\begin{gather}
\label{ineqtildeytildeyL2L2}
||(\tilde y\cdot \nabla)\tilde{y}||_{L^2((0,T),L^2(\Omega)^3)}\leqslant ||\nabla \tilde{y}||_{L^\infty((0,T),L^2(\Omega)^9)}||\tilde{y}||_{L^2((0,T),L^\infty(\Omega)^3)}.
\end{gather}
 From \eqref{ineqnabla}, \eqref{ineqtildeyL2Linfty} and \eqref{ineqtildeytildeyL2L2}, we obtain
$$||(\tilde y \cdot \nabla)\tilde{y}||_{L^2((0,T),L^2(\Omega)^3)}\leqslant C ||(y,p,v)||^2_E,$$
which, together with \eqref{deftildeyexp}, gives \eqref{tobeprovedforC}.
\end{proof}

We now consider the element $e_0=(0,0,0)$ and we compute $$d \mathcal F(e_0)(y,q,v)
=y_t-\Delta y+(\overline y\cdot \nabla)y+(y\cdot \nabla)\overline y +\nabla p-(0,0,v).$$
Proposition~\ref{cor-une-seule-composante} implies that this application is onto. Hence,
taking $g=(0,y^0)$ and applying Proposition~\ref{prop-inverse-mapping},
Theorem~\ref{th-STLC} easily follows (in particular because the trajectory $\overline y$ can be chosen as small as we want since $\varepsilon$ can be arbitrarily small).

\section*{Acknowledgments}
The authors would like to thank Sergio Guerrero for fruitful discussions concerning Proposition~\ref{cor-regular-controls}.
\appendix
\section{Appendix: Creation of the matrix \texorpdfstring{${L_0}$}{}}\label{program}
In this appendix, we explain how the matrix ${L_0}$ at point $\xi^0$ (which
represents all the differentiated equations of System \eqref{syst1} up to the
order $19$) was created.
The program is written in $C^{++}$, using the library uBLAS which is well-adapted to the manipulation of sparse matrices. It is a parallel openMP algorithm,
using 8 cores. We are not going to give all the technical details but
just explain rapidly the spirit of the algorithm.
To simplify, we will assume that the following ``black boxes'' (that had to be created) are at our disposal:
\begin{enumerate}
\item An evaluation function {\tt  ep } which evaluates a polynomial (represented by a vector) at $\xi^0$. This evaluation function can be created so that it can verify that $\xi^0$ is not a root of the polynomial $P(0,.,.)$. (one just has to see if the evaluation is equal to $0$ whereas the polynomial has nonzero coefficients).
\item A derivation function {\tt deqex } which differentiates an equation of level $m$ with respect to $x_1,x_2
,x_3$ or $t$.
\end{enumerate}
A partial differential equation which is a derivative of order $m$ of some of
the equations of \eqref{syst1} will be represented in a
matricial form in the following way: We know that there are at most $F(m+3)$ derivatives appearing, and we observe that the coefficients
are polynomials in $(x_1,x_2,x_3)$ of an order less than $4$ (it is a vector
space of dimension $35$). Hence an equation of order $m$ is represented by a matrix
with $F(m+3)$ lines and $35$ columns, where on each line one can find the
coefficient of the partial derivatives of $z^1$ (or $z^2$ appearing)
corresponding to the number of this line, thanks to the natural bijection
between $\mathbb N^4$ and $\mathbb N$.
Since we have $3$ equations in \eqref{syst1} and $2$ unknowns ($z^1$ and $z^2$), one can write the matrix $M$ in the following way:
\begin{gather}\label{dec-mat}\begin{pmatrix}
A_1&B_1\\A_2&B_2\\A_3&B_3
\end{pmatrix}.\end{gather}
For $i=1,2,3$, $A_i$ represents the derivatives of $z^1$ appearing in the derivatives of the $i$-th equation of \eqref{syst1} and $B_i$ those of $z^2$. Hence, we can compute these $A_i$ and $B_i$ separately and then gather them to obtain ${L_0^0}$.

The algorithm is the following. We explain it for the first equation of \eqref{syst1} and for the unknown $z^1$ (i.e. for $A_1$, but it is the same for the other matrices).
\begin{enumerate}
\item We create a matrix $e$ that represents the equation. We use \emph{ep} to fill the line of ${L_0^0}$ corresponding to the equation in a .txt file under the form $i$ $j$ $A_1(i,j)$. We create a matrix $h$ which is empty for the moment. In fact in $e$ we will keep the equations of level $m-1$ and in $h$ we will fill the equations of level $m$.
\item We create a ``for'' loop on $m$ which will represent the level of equations we are creating. The integer $m$ goes from $1$ to $19$ since we differentiate $19$ times at most.
\item We create a second ``for'' loop in the interior of the first loop on a number $n$ which represents one of the equations of level $m$. Thanks to the definition of the function $F$ given in Subsection~\ref{denomb}, we have $F(m-1)+1\leqslant n\leqslant F(m)$. If $m=1$, then $n$ goes from $F(0)+1=2$ to $F(1)=5$ ($n$ represents $\partial_1$, $\partial_2$, $\partial_3$ or $\partial_t$). If $m=2$, then $n$ goes from $F(1)+1=6$ to $F(2)=15$ ($n$ represents $\partial^2_{11},\partial^2_{12},\partial^2_{13},\partial^2_{1t},\partial^2_{22},\partial^2_{23},\partial^2_{2t},\partial^2_{33},\partial^2_{3t}$ or $\partial^2_{tt}$), etc. This loop is parallelized on our $8$ cores. In this loop, we want to create the $n$-th equation denoted $E_n$, which is of level $m$. Hence we take a suitable equation of level $m-1$ denoted $E_r$ which is so that if we differentiate $E_r$ with respect to $1,2,3$ or $t$, we obtain $E_n$. For example, if we consider $m=2$ and if we want to obtain the first equation of \eqref{syst1} differentiated two times with respect to $1$, then we consider the equation $E_r$ to be the first equation of \eqref{syst1} differentiated one time with respect to $1$ and differentiated with respect to $1$ to obtain $E_n$.
\item Once the loop on $n$ is ended, we have in our matrix $e$ all the equations of level $m-1$ and in $h$ we have just created all the equations of level $m$. Now we just have to use our evaluation function \emph{ep} on $h$ to obtain the coefficients of the lines of $A_1$ corresponding to the equations that are of level $m$, i.e. the equations numbered from $F(m-1)+1$ to $F(m)$. We write these coefficients in our .txt file under the form $i$ $j$ $A_1(i,j)$.
\item We update now $e$, take $e=h$, we empty $h$ and we can go to the following loop $m+1$.
\end{enumerate}
At the end we have created a file containing the coefficients of a sparse matrix $A_1$ of size $(8855,14950)$. Using the same program with $z^2$ and the two others equation we obtain five other files representing five matrices that we gather as in \eqref{dec-mat} to obtain the matrix ${L_0}(\xi^0)={L_0^0}$. Our matrix ${L_0^0}$, which represents all the equations, is of size $(30360,29900)$ and has $651128$ nonzero coefficients. Only $0.07\%$ of the coefficients are different from $0$, with an average of $21.44$ nonzero coefficients on each row, which is logical since we are working with coefficients that are polynomials of small degree, so we do not create many terms on each line when we differentiate the equations.
In the following figure, one can observe how the nonzero coefficients of ${L_0^0}$ are distributed.
\begin{figure}[!ht]
\label{figure-non=zero-L00}
\begin{center}
\includegraphics[scale=0.55]{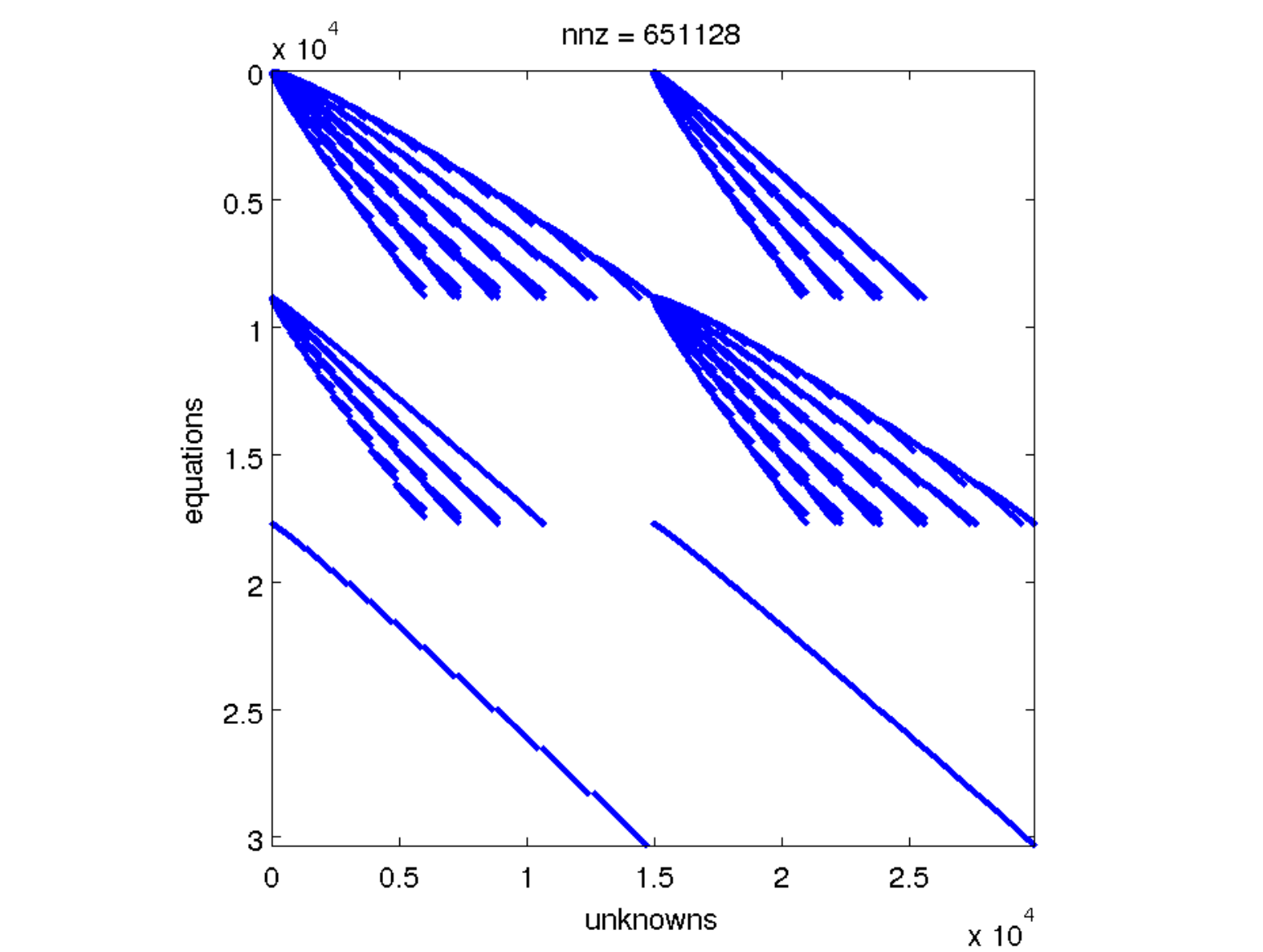}
\caption{Distribution of the nonzero coefficients of ${L_0^0}$.}
\end{center}
\end{figure}

\bibliographystyle{plain}
\bibliography{Biblio}
\end{document}